\documentclass[]{article}

\usepackage{setspace} 
\usepackage[backend=bibtex, style=alphabetic, minalphanames=10, maxalphanames=10, maxbibnames=10]{biblatex}

\usepackage{amsmath}
\usepackage{amsthm}
\usepackage{thmtools}
\usepackage{amssymb}
\usepackage{enumitem}
\usepackage{tikz-cd}
\usepackage{mathtools}
\usepackage{cleveref}
\usepackage{xcolor}
\usepackage{import}
\usepackage{graphicx}

\usepackage[a4paper, portrait]{geometry}
\newtheorem{theorem}{Theorem}
\numberwithin{theorem}{section} 
\newtheorem{lemma}[theorem]{Lemma}
\numberwithin{lemma}{section} 
\newtheorem{corollary}[theorem]{Corollary}
\numberwithin{corollary}{section} 
\newtheorem{proposition}[theorem]{Proposition}
\numberwithin{proposition}{section} 
\newtheorem{claim}[theorem]{Claim}
\numberwithin{claim}{section} 
\theoremstyle{definition}
\newtheorem{definition}[theorem]{Definition}
\numberwithin{definition}{section} 
\theoremstyle{definition}

\numberwithin{notation}{section} 
\theoremstyle{definition}
\newtheorem{remark}[theorem]{Remark}
\numberwithin{remark}{section} 
\theoremstyle{definition}

\numberwithin{example}{section} 
\theoremstyle{definition}
\newtheorem{question}[theorem]{Question}
\numberwithin{question}{section} 
\newtheorem{conjecture}[theorem]{Conjecture}
\numberwithin{conjecture}{section} 
\newtheorem{fact}[theorem]{Fact}
\numberwithin{fact}{section}

\newcommand{\uq}[1]{\mathcal{U}_q({#1})}
\newcommand{\Sl}[1]{\mathfrak{sl}_{#1}}

\newcommand{\Arr}{\mathrm{Arr}}
\newcommand{\Ed}{\mathrm{Ed}}
\newcommand{\tor}{\mathrm{tor}}

\newcommand{\Sym}{\mathrm{Sym}}
\newcommand{\hmu}{\hat{\mu}}
\newcommand{\uqslinf}{\uq{\widehat{\Sl{\infty}}}}
\newcommand{\uqsl}[1]{\uq{\widehat{\Sl{#1}}}}
\newcommand{\uqhat}[1]{\uq{\widehat{\mathfrak{#1}}}}
\newcommand{\uqtor}[1]{\uq{{\Sl{{#1,\tor}}}}}
\newcommand{\Oint}{\mathcal{O}_{\mathrm{int}}}
\newcommand{\preqchoose}{\genfrac{[}{]}{0pt}{}}
\newcommand{\qchoose}[2]{\preqchoose{#1}{#2}_q}
\newcommand{\Yuq}[1]{\mathcal{Y}(\uqhat{#1})}

\newcommand{\Ytor}[1]{\mathcal{Y}(\uqtor{#1})}
\newcommand{\Auq}[1]{\tilde{A}(\uqhat{#1})}

\newcommand{\Ysub}[1]{\mathcal{Y}_{#1}}
\newcommand{\Asub}[1]{\tilde{A}_{#1}}

\bibliography{ref.bib}

\title{Folding of cluster algebras and quantum toroidal algebras}
\author{Lior Silberberg}
\date{}

\begin{document}

\maketitle

\begin{abstract}
	In this paper, we study the relationship between the representation theory of the quantum affine algebra \(\uqslinf\) of infinite rank, and that of the quantum toroidal algebra \(\uqtor{2n}\). Using monoidal categorifications due to Hernandez-Leclerc and Nakajima, we establish a cluster-theoretic interpretation of the folding map \(\phi_{2n}\) of \(q\)-characters, introduced by Hernandez. To this end, we introduce a notion of foldability for cluster algebras arising from infinite quivers and study a specific case of cluster algebras of type \(A_\infty\). \par
	Using this interpretation of \(\phi_{2n}\), we prove a conjecture of Hernandez in new cases. Finally, we study a particular simple \(\uqtor{2n}\)-module whose \(q\)-character is not a cluster variable, and conjecture that it is imaginary. 
\end{abstract}

\setcounter{tocdepth}{1} 
\tableofcontents

\section{Introduction}
In the 1980s, Drinfeld and Jimbo associated to any complex symmetrizable Kac--Moody algebra \(\mathfrak{g}\) a Hopf algebra \(\uq{\mathfrak{g}}\) with generic \(q \in \mathbb{C}^*\), called the quantum Kac--Moody algebra associated to \(\mathfrak{g}\). In turn, to \(\uq{\mathfrak{g}}\), one can associate an affinization \(\uqhat{g}\). It is a \(\mathbb{C}\)-algebra which is in general not known to be a Hopf algebra. Nonetheless, the category \(\Oint(\uqhat{\mathfrak{g}})\) of integrable \(\uqhat{g}\)-modules admits a fusion product, which equips its Grothendieck group with a commutative ring structure, see \cite{H05,H07}. \par 
One important notion in the representation theory of affinized quantum Kac--Moody algebras is that of \(q\)-characters, first introduced by Frenkel and Reshetihkin in \cite{FR} and extended to the general case in \cite{H05}. In brief, the \(q\)-character of a module \(V\) in \(\Oint(\uqhat{g})\) is (in general) an infinite sum that encodes the dimensions of its \(l\)-weight spaces. \par
For a complex finite-dimensional simple Lie algebra \(\mathfrak{g}_\mathrm{fin}\), let \(\mathfrak{g}\) be the associated affinized Lie algebra. It turns out that \(\uq{\mathfrak{g}} \simeq \uq{\hat{\mathfrak{g}}_{\mathrm{fin}}}\), and these are the ``smallest" examples of affinized Kac--Moody algebras. We call the algebras arising in this way \emph{quantum affine algebras}. They have been studied extensively in the past few decades, see for example \cite{CP91,CP95,FR,N01,FM01}. \par
For an affine Lie algebra \(\mathfrak{g}\), however, the ``double affine" quantum algebra \(\uqhat{g}\) is not isomorphic to any quantum Kac--Moody algebra. We call these algebras \emph{quantum toroidal algebras}. They were first introduced in type \(A\) in \cite{GKV}. Their structure and representation theory are considerably more complex than those of the quantum affine algebras. For example, only recently a tensor product for their integrable representations was defined, see \cite{Lau25}. \par
One may ask the following natural question:
\begin{question}\label{main-question}
	Can we use the well studied representation theory of quantum affine algebras in order to understand the representation theory of quantum toroidal algebras?
\end{question}
A first result in this direction was given in \cite{H-folding}. In this paper, Hernandez regards the Dynkin diagram of affine type \(A_n^{(1)}\) as a folding of the bi-infinite Dynkin diagram \(A_\infty\) (infinite in both directions). Although there is no direct way to ``fold" the quantum affine algebra \(\uqslinf\) onto the quantum toroidal algebra \(\uqtor{n+1} \coloneq \uq{\widehat{\widehat{\Sl{n+1}}}}\), he proves that there is a homomorphism \(\phi_n\) that ``folds" \(q\)-characters of integrable \(\uqslinf\)-modules onto virtual \(q\)-characters of integrable \(\uqtor{n+1}\)-modules. Moreover, he proves that \(\phi_n\) maps \(q\)-characters of Kirillov--Reshetikhin \(\uqslinf\)-modules to \(q\)-characters of Kirillov--Reshetikhin \(\uqtor{n+1}\)-modules. Finally, he gives the following conjecture.
\begin{conjecture}[Conjecture 5.3 in \cite{H-folding}]\label{conjecture-intro}
	The homomorphism \(\phi_n\) maps \(q\)-characters of integrable \(\uqslinf\)-modules to \(q\)-characters of actual (non-virtual) integrable \(\uqtor{n+1}\)-modules. 
\end{conjecture}
The idea of relating the \(q\)-character theory of different quantum groups appears also in \cite{H10kr, W23}, where it was used to obtain new multiplication formulas. Similar ideas not directly involving \(q\)-characters can be found in \cite{KKKO}. \par
An important notion in the representation theory of quantum affine algebras (and more generally affinized Kac--Moody algebras) is that of \emph{monoidal categorification}. It was introduced by Hernandez and Leclerc in \cite{HL10}, and studied further in \cite{N11, HL13, Q17, BriCha19, KKOP, KKOP2}, to name a few. In brief, the idea is to study an abelian monoidal category \(\mathcal{C}\) by identifying its Grothendieck ring \(\mathcal{R}(\mathcal{C})\) with a cluster algebra \(\mathcal{A}\), in such a way that cluster monomials in \(\mathcal{A}\) correspond to classes of real simple objects in \(\mathcal{C}\). An object is said to be a real simple object if its tensor product with itself is still simple. The tensor (or fusion) product of integrable \(\uq{\mathfrak{g}}\)-modules provides many examples of categories which are monoidal categorifications of cluster algebras. \par
In this paper, we build on the results of \cite{H-folding} and establish that we can understand \(\phi_n\) in a cluster-theoretic way. We first introduce a category \(\mathcal{C}_1\) of \(\uqslinf\)-modules which gives a monoidal categorification of a cluster algebra \(\mathcal{A}_\infty\), associated to an ice quiver \(\Gamma_\infty\) with principal part of type \(A_\infty\) with alternating orientation. We then introduce a notion of folding of cluster algebras suitable for (locally finite) infinite quivers with a group action, which we call \emph{strong global foldability}. With this concept in hand, we get special sets of cluster variables which we call \emph{orbit-clusters}, which correspond to the clusters of the folded cluster algebra. Our first main result is the following: let \(Q\) be a quiver whose underlying diagram is of type \(A_{n}^{(1)}\) (so it is an orientation of an \((n+1)\)-cycle), and let \(A_Q\) be the orientation of \(A_\infty\) corresponding to \(Q\) (see \Cref{section_folding_Ainf} for details). Let also \(G\) be the group of permutations of \((A_\infty)_0\) generated by shifting the vertices by \(n+1\), so that \(A_Q^G = Q\).
\begin{theorem}\label{folding-theorem-intro}
	In the setting above, the quiver \(A_Q\) is strongly globally foldable with respect to \(G\) if and only if \(Q\) is not cyclically oriented.  
\end{theorem}
This first result suggests that there is a strong connection between the cluster algebra of type \(A_\infty\) and the cluster algebras of type \(A_{n}^{(1)}\). \par 
In \cite{N11}, Nakajima constructs a category \(\mathcal{C}_1\) for any simply-laced symmetric affinized quantum Kac--Moody algebra, whose Dynkin diagram has no odd cycles. He shows it provides a monoidal categorifications of an associated cluster algebra. For the particular case of \(\uqtor{2n}\), the monoidal categorification is of a cluster algebra \(\mathcal{A}_{2n}\) associated to an ice quiver \(\Gamma_{2n}\), whose principal part is of type \(A_{2n-1}^{(1)}\) with alternating orientation. \par
We show that there is a group \(G\) acting on the vertices of \(\Gamma_\infty\), such that \(\Gamma_\infty^G = \Gamma_{2n}\). We then prove the following result, which is the main theorem of the paper (in what follows, we identify classes of modules in the Grothendieck ring with their \(q\)-characters). 
\begin{theorem}\label{theorem-ice-folding-intro}
	The ice quiver \(\Gamma_\infty\) is strongly globally foldable with respect to \(G\). \par
	If \(\chi_q(L)\) is a cluster variable belonging to an orbit-cluster of \(\mathcal{A}_\infty\), then \(\phi_{2n}(\chi_q(L))\) is a cluster variable of \(\mathcal{A}_{2n}\). Each cluster variable of \(\mathcal{A}_{2n}\) can be obtained in this way. \par
	In particular, \(\phi_{2n}(\chi_q(L))\) is the \(q\)-character of a simple \(\uqtor{2n}\)-module.
\end{theorem}
The theorem has the following consequences:
\begin{enumerate}
	\item The map \(\phi_{2n}\) is more than just a projection of \(q\)-characters, it relates the cluster structures associated to \(\uqslinf\) and \(\uqtor{2n}\). 
	\item We can use some known multiplicative formulas for \(q\)-characters of \(\uqslinf\) to deduce analogous ones for \(\uqtor{2n}\). 
	\item If \(L\) is a simple \(\uqslinf\)-module whose class belongs to an orbit-cluster in \(\mathcal{A}_\infty\), then \(\phi_{2n}(\chi_q(L))\) is the \(q\)-character of a simple \(\uqtor{2n}\)-module. In particular, we prove \Cref{conjecture-intro} in this case. These are the first examples of modules satisfying \Cref{conjecture-intro} outside the Kirillov--Reshetikhin case.
\end{enumerate} 
As an application, we compute some \(q\)-characters for \(\uqtor{2n}\)-modules that do not correspond to cluster variables. We obtain a few more examples of non-Kirillov--Reshetikhin modules that satisfy \Cref{conjecture-intro}. As a particular case, we study a simple \(\uqtor{2n}\)-module, which we denote by \(L(\overline{m_{1,2n}})\), and we conjecture the following.
\begin{conjecture}\label{imag-conj-intro}
	The simple \(\uqtor{2n}\)-module \(L(\overline{m_{1,2n}})\) is an imaginary module, that is, the fusion product \(L(\overline{m_{1,2n}})\ast L(\overline{m_{1,2n}})\) is not simple. 
\end{conjecture}
For \(n=1\), this module has been studied in Example 6.15 of \cite{N11}, where it was shown to be imaginary. We show that \Cref{conjecture-intro} implies \Cref{imag-conj-intro}.

\subsection*{Structure of the paper}
The paper is organized as follows. 
\begin{itemize}
	\item \Cref{Preliminaries} gives the necessary background on cluster algebras and quantum groups. 
	\item \Cref{uqslinf} recalls useful results on \(\uqslinf\).
	\item \Cref{folding} studies the folding of infinite quivers and their associated cluster algebras. 
	\item \Cref{section_folding_Ainf} contains the proof of \Cref{folding-theorem-intro}. 
	\item \Cref{C1infty} discusses the monoidal categorification coming from a category \(\mathcal{C}_1\) of \(\uqslinf\)-modules. 
	\item \Cref{uqtor} uses the results of \cite{N11} and those of the previous sections to give a proof for \Cref{theorem-ice-folding-intro}.
	\item \Cref{further} applies \Cref{theorem-ice-folding-intro} to study the \(q\)-characters of certain \(\uqtor{2n}\)-modules that are not cluster variables, discusses \Cref{imag-conj-intro}, and outlines further questions.
\end{itemize}

\subsection*{Acknowledgments}
I would like to thank my advisors David Hernandez and Yann Palu for their guidance and support. This project has received funding from the European Union’s Horizon Europe research and innovation programme under the Marie Skłodowska-Curie grant agreement n° 101126554. \par 
Co-Funded by the European Union. Views and opinions expressed are however those of the author only and do not necessarily reflect those of the European Union. Neither the European Union nor the granting authority can be held responsible for them. \def\svgwidth{3.6em}
\begingroup%
  \makeatletter%
  \providecommand\color[2][]{%
    \errmessage{(Inkscape) Color is used for the text in Inkscape, but the package 'color.sty' is not loaded}%
    \renewcommand\color[2][]{}%
  }%
  \providecommand\transparent[1]{%
    \errmessage{(Inkscape) Transparency is used (non-zero) for the text in Inkscape, but the package 'transparent.sty' is not loaded}%
    \renewcommand\transparent[1]{}%
  }%
  \providecommand\rotatebox[2]{#2}%
  \newcommand*\fsize{\dimexpr\f@size pt\relax}%
  \newcommand*\lineheight[1]{\fontsize{\fsize}{#1\fsize}\selectfont}%
  \ifx\svgwidth\undefined%
    \setlength{\unitlength}{2167.83728027bp}%
    \ifx\svgscale\undefined%
      \relax%
    \else%
      \setlength{\unitlength}{\unitlength * \real{\svgscale}}%
    \fi%
  \else%
    \setlength{\unitlength}{\svgwidth}%
  \fi%
  \global\let\svgwidth\undefined%
  \global\let\svgscale\undefined%
  \makeatother%
  \begin{picture}(1,0.18682214)%
    \lineheight{1}%
    \setlength\tabcolsep{0pt}%
    \put(0,0){\includegraphics[width=\unitlength,page=1]{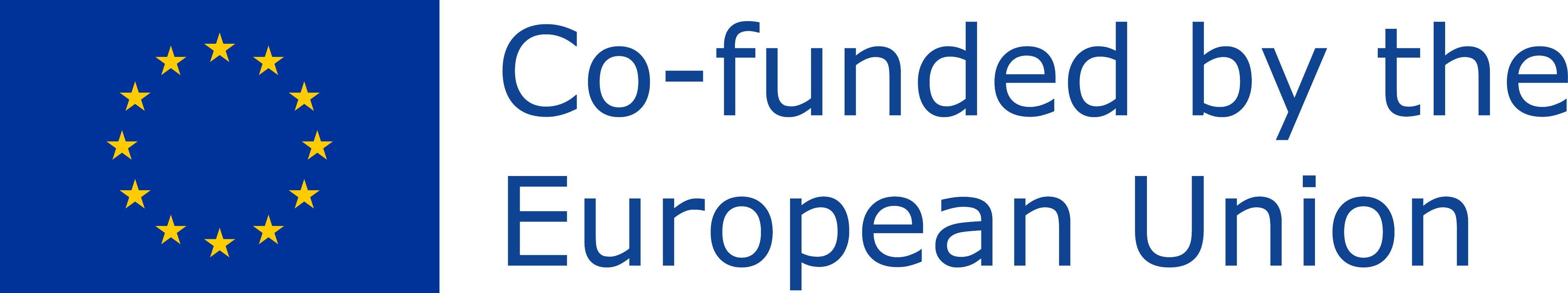}}%
  \end{picture}%
\endgroup%

\section{Preliminaries}\label{Preliminaries}

\subsection{Infinite quivers and cluster algebras}

This section covers the basics on cluster algebras and sets some notation we will need for this paper. A more detailed treatment can be found, for example, in \cite{FZ02,FWZ24} or in the survey \cite{KEl08}.

\begin{definition}
	A quiver \(Q\) is a directed graph with vertex set \(Q_0\) and arrow set \(Q_1\), equipped with two functions \(s,t : Q_1 \rightarrow Q_0\). Given \(\alpha \in Q_1\), we call \(s(\alpha)\) the source of \(\alpha\), and \(t(\alpha)\) the target of \(\alpha\).
\end{definition}
For a quiver \(Q\) and two subsets \(X,Y \subseteq Q_0\), we set 
\begin{align*}
	\Arr_Q(X,Y) \coloneq \#\{\alpha \in Q_1 \mid s(\alpha) \in X, t(\alpha) \in Y\} \in \mathbb{N} \cup \{\infty\}
\end{align*}
the number of arrows from \(X\) to \(Y\), and 
\begin{align*}
	\Ed_Q(X,Y) = \Arr_Q(X,Y) + \Arr_Q(Y,X)
\end{align*}
the number of edges between \(X\) and \(Y\). If \(X = \{x\}\) we will write \(\Arr_Q(x,Y)\) for \(\Arr_Q(X,Y)\) and \(\Ed_Q(x,Y)\) for \(\Ed_Q(X,Y)\) (and similarly for \(Y = \{y\}\)). \par
To describe a quiver \(Q\) up to isomorphism, it is thus enough to give its vertex set \(Q_0\) and \(\Arr_Q(x,y)\) for every pair of vertices \(x,y\in Q_0\). \par
In this paper, we assume our quivers do not have loops or \(2\)-cycles. This means \(\Arr_Q(x,y)\cdot \Arr_Q(y,x) = 0\) for all \(x,y \in Q_0\). We allow a quiver \(Q\) to have an infinite countable vertex set. However, we require it to be locally finite, that is, \(\Ed_Q(x, Q_0) < \infty\) for every \(x \in Q_0\). \par
An ice quiver \((Q,F)\) is a quiver \(Q\) together with a subset of its vertices \(F \subseteq Q_0\). We call the vertices in \(F\) frozen, and those in \(Q_0 \setminus F\) mutable. For simplicity, and given our application in the following sections, we assume there are no arrows between frozen vertices. Any quiver \(Q\) can be regarded as an ice quiver by setting \(F=\emptyset\). \par
Given an ice quiver \((Q,F)\) and a mutable vertex \(z \in Q_0\), we can define a new ice quiver \((\mu_z (Q), F)\) as follows. The vertex set of \(\mu_z (Q)\) is \(Q_0\). The arrows are given by
\begin{align*}
	\Arr_{\mu_z (Q)}(x,y) = \begin{cases}
		0 & \text{ if } x,y \in F \\
		\max\{0,b_{x,y}-b_{y,x}\} & \text{ if } x,y \notin F \cup \{z\} \\ 
		\Arr_Q(y,x) & \text{ otherwise}
	\end{cases},
\end{align*}
where \(b_{x,y} = \Arr_Q(x,y) + \Arr_Q(x,z) \cdot \Arr_Q(z,y)\). We call the new ice quiver the mutation of \((Q,F)\) at the vertex \(z\). We note that \(\mu_{z}(\mu_z(Q)) \simeq Q\). \par
Having defined the mutation operation, we can associate to an ice quiver \((Q,F)\) a commutative algebra as follows. We fix \((Q,F)\), and consider the field
\begin{align*}
	\tilde{\mathcal{F}} = \mathbb{Q}(x_v \mid v \in Q_0),
\end{align*}
where \(x_v\) for \(v \in Q_0\) are indeterminates. A seed \(((Q',F'),\underline{y})\) is a pair of an ice quiver \((Q',F')\), and \(\underline{u} = \{u_v \in \tilde{\mathcal{F}} \mid v \in Q_0\}\), a free generating set of \(\tilde{\mathcal{F}}\). \par
We extend the mutation operations to seeds. For a mutable \(z \in Q_0'\), we define the mutation of the seed \(((Q',F'),\underline{u})\) at \(z\) to be the seed \(((\mu_z (Q'),F'),\mu_z (\underline{u}))\), where \(\mu_z (\underline{u}) = \underline {u} \cup \{u'_z\} \setminus \{u_z\} \) and 
\begin{align*}
	u'_z = \frac{\prod_{w \in Q'_0} u_w^{\Arr_{Q'}(z,w)}+\prod_{w \in Q'_0} u_w^{\Arr_{Q'}(w,z)}}{u_z}.
\end{align*}
Notice that \(\mu_z(\mu_z (\underline{u})) = \underline{u}\). \par
To the ice quiver \((Q,F)\), we associate the set \(\underline{x} = \{x_v \mid v \in Q_0\}\). We call \(((Q,F),\underline{x})\) the initial seed. We let \(\mathcal{S}\) be the set of all seeds obtained from the initial seed by a finite sequence of mutation. Let \(s \in \mathcal{S}\) be some seed, so \(s = ((Q_s,F_s),\underline{u}_s)\). An easy observation is that \((Q_s)_0 = Q_0\) and \(F_s = F\). Moreover, \(\underline{u}_s\) contains all \(x_v\) such that \(v \in F\).  We call \(\underline{u}_s\) a cluster, and any \(u_v \in \underline{u}_s\) a cluster variable.
\begin{definition}
	The cluster algebra \(\mathcal{A}_{(Q,F)}\) is the subring of \(\tilde{\mathcal{F}}\) generated by all cluster variables. That is \(\mathcal{A}_{(Q,F)} \coloneq \mathbb{Z}[\mathcal{X}] \subseteq \tilde{\mathcal{F}}\) for \(\mathcal{X} = \bigcup_{s \in \mathcal{S}} \underline{u}_s\).
\end{definition}
The Laurent phenomenon is a very well known result which states that every cluster variable is a Laurent polynomial in the initial cluster with integral coefficient. In our context, we set \(\mathcal{F} = \mathbb{Z}[x_v \mid v\in F][x_v^{\pm 1} \mid v \in Q_0 \setminus F]\), so the Laurent phenomenon can be stated as \(\mathcal{A}_{(Q,F)} \subseteq \mathcal{F}\).

\subsection{Symmetric Kac-Moody quantum groups and their affinizations}\label{Section-quantum-groups}

We give here some necessary background on quantum groups and their representation theory. For the sake of concreteness, we restrict our attention to symmetric Cartan datum. For proofs and more details, see \cite{CP91,CP95,N01,H05}. \par
Let \(I\) be a finite index set, and \(\lvert I \rvert = n\). Let \(C = (C_{i,j})_{i,j \in I}\) be a symmetric matrix such that \(C_{i,i} = 2\) and \(C_{i,j} \in \mathbb{Z}_{\leq 0}\) for \(i\neq j\). To this information, we associate a realization \((\mathfrak{h},\Pi,\Pi^\vee)\)  (see \cite{K90}), with \(\mathfrak{h}\) a complex vector space of dimension \(2n-\mathrm{rank}(C)\), the set of simple roots \(\Pi = \{\alpha_i \mid i \in I\}\subseteq \mathfrak{h}^*\), and the set of simple coroots \(\Pi^\vee = \{\alpha_i^\vee \mid i \in I\} \subseteq \mathfrak{h}\), such that \(\alpha_j(\alpha_i^\vee) = C_{i,j}\). We let \(\Delta\) be the associated root system and we set \(Q^+ = \sum_{i\in I} \mathbb{N}\alpha_i\). We denote by 
\begin{align*}
	P^\vee = \{h \in \mathfrak{h} \mid \alpha_i(h) \in \mathbb{Z},\ i \in I\},
\end{align*}
the group of integral coweights, by
\begin{align*}
	P = \{\lambda \in \mathfrak{h}^* \mid \lambda(\alpha_i^\vee) \in \mathbb{Z},\ i \in I\}
\end{align*}
the group of integral weights, and by
\begin{align*}
	P^+ = \{\lambda \in \mathfrak{h}^* \mid \lambda(\alpha_i^\vee) \in \mathbb{N},\ i \in I\}
\end{align*}
its submonoid of dominant integral weights. \par
From now on we fix \(q \in \mathbb{C}^*\), not a root of unity. Given \(a,b \in \mathbb{N}\) such that \(a\leq b\) we set 
\begin{align*}
	[b]_q = \frac{q^b-q^{-b}}{q-q^{-1}}, && [b]_q! = [b]_q\cdots[1]_q, && \qchoose{b}{a} = \frac{[b]_q!}{[b-a]_q![a]_q!}.
\end{align*}
\begin{definition}
	The quantum Kac-Moody algebra \(\uq{\mathfrak{g}}\) is the \(\mathbb{C}\)-algebra with generators \(k_h\) and \(x_{i}^\pm\) with \(h \in P^\vee\) and \(i \in I\), subject to relations
	\begin{gather*}
		k_h k_{h'} = k_{h+h'}, \qquad k_0 = 1, \\
		k_h x_j^\pm k_{-h} = q^{\pm \alpha_j(h)} x_j^\pm, \\
		[x_i^+,x_j^-] = \delta_{i,j} \frac{k_{\alpha_i^\vee} - k_{-\alpha_i^\vee}}{q-q^{-1}}, \\
		\sum_{r=0}^{1-C_{i,j}}(-1)^r\qchoose{1-C_{i,j}}{r}(x_{i}^\pm)^{1-C_{i,j}-r}x_j^\pm (x_i^\pm)^r = 0.
	\end{gather*}
\end{definition}
Let \(\uq{\mathfrak{g}}^+\), \(\uq{\mathfrak{g}}^-\) and \(\uq{\mathfrak{h}}\) be the subalgebras of \(\uq{\mathfrak{g}}\) generated by the \(x_i^+\), \(x_i^-\) and \(k_h\), respectively. We have the triangular decomposition \(\uq{\mathfrak{g}} \simeq \uq{\mathfrak{g}}^-\otimes \uq{\mathfrak{h}} \otimes \uq{\mathfrak{g}}^+\).
\begin{definition}\label{qaffinization-def}
	\allowdisplaybreaks The quantum affinization of \(\uq{\mathfrak{g}}\) is the \(\mathbb{C}\)-algebra \(\uqhat{g}\) with generators \(k_h\), \(h_{i,m}\) and \(x_{i,r}^\pm\) with \(h\in P^\vee\), \(i \in I\), \(r\in\mathbb{Z}\) and \(m \in \mathbb{Z}_{\neq 0}\), subject to relations
	\begin{gather*}
		k_h k_{h'} = k_{h+h'}, \qquad k_0 = 1, \qquad [k_h,h_{i,m}] = 0, \qquad [h_{i,m},h_{j,m'}] = 0, \\
		k_h x_{j,r}^\pm k_{-h} = q^{\pm\alpha_j(h)}x_{j,r}^\pm, \\
		[h_{i,m},x_{j,r}^\pm] = \pm \frac{1}{m}[mC_{i,j}]_q x_{j,m+r}^\pm, \\
		[x_{i,r}^+,x_{j,r'}^-] = \delta_{i,j} \frac{\phi_{i,r+r'}^+-\phi_{i,{r+r'}}^-}{q-q^{-1}}, \\
		x_{i,r+1}^\pm x_{j,r'}^\pm - q^{\pm C_{i,j}} x_{j,r'}^\pm x_{i,r+1}^\pm = q^{\pm C_{i,j}} x_{i,r}^\pm x_{j,r'+1}^\pm - x_{j,r'+1}^\pm x_{i,r}^\pm, \\
		\sum_{\pi \in S_d}\sum_{k=0}^{d}(-1)^k \qchoose{d}{k} x^{\pm}_{i,r_{\pi(1)}}\cdots x^{\pm}_{i,r_{\pi(k)}} x_{j,r'}^\pm x^{\pm}_{i,r_{\pi(k+1)}}\cdots x^{\pm}_{i,r_{\pi(d)}} = 0,
	\end{gather*}
	where \(d=1-C_{i,j}\), \(r_1,...,r_{d}\) is any sequence of integers, and \(\phi_{i,m}^\pm\) are determined by
	\begin{align*}
		\sum_{m=0}^\infty \phi_{i,\pm m }^\pm z^{\pm m} = k_{\pm \alpha_i^\vee} \exp\left(\pm(q-q^{-1})\sum_{m'=1}^\infty h_{i,\pm m'} z^{\pm m'}\right),
	\end{align*}
	and \(\phi_{i,\pm m}^\pm = 0\) for \(m < 0\).
\end{definition}
Let \(\uqhat{g}^+\), \(\uqhat{g}^-\) and \(\uqhat{h}\) be the subalgebras of \(\uqhat{g}\) generated by the \(x_{i,r}^+\), \(x_{i,r}^-\) and \(k_h, h_{i,m}\), respectively. We have a triangular decomposition for \(\uqhat{g}\), that is, a linear isomorphism \(\uqhat{g} \simeq \uqhat{g}^-\otimes \uqhat{h} \otimes \uqhat{g}^+\). \par
The subalgebra of \(\uqhat{g}\) generated by \(x_i^\pm\) and \(k_h\) is isomorphic to \(\uq{\mathfrak{g}}\). Identifying these, we have \(\uq{\mathfrak{g}}^\pm \subseteq \uqhat{g}^\pm\) and \(\uq{\mathfrak{h}} \subseteq \uqhat{h}\). \par
A \(\uq{\mathfrak{h}}\)-module \(V\) is said to be \(\uq{\mathfrak{h}}\)-diagonalizable if
\begin{align*}
	V = \bigoplus_{\omega \in P} V_{\omega}, \qquad V_{\omega} \coloneq \{v \in V \mid k_h(v) = q^{\omega(h)}v,\ \forall h\in P^\vee\}.
\end{align*}
A \(\uq{\mathfrak{g}}\)-modules \(V\) is called integrable if it is \(\uq{\mathfrak{h}}\)-diagonalizable, its weight spaces \(V_\omega\) are all finite-dimensional, and for any \(\omega \in P\) and \(i \in I\), we have \(V_{\omega \pm r\alpha_i}=0\) for \(r\) large enough. \par
\begin{definition}\label{Oint-def}
	The category \(\Oint(\uqhat{g})\) is the full subcategory of \(\uqhat{g}\)-modules whose objects are the modules \(V\) satisfying
	\begin{enumerate}
		\item \(V\) is integrable as a \(\uq{\mathfrak{g}}\)-module,
		\item there are a finite number of weights \(\lambda_1,...,\lambda_s \in \mathfrak{h}^*\) for which
		\begin{align*}
			V_\omega \neq 0 \implies \omega \in \bigcup_{j=1}^{s} (\lambda_j - Q^+). 
		\end{align*}
	\end{enumerate}
\end{definition}
A pair \((\lambda,\gamma)\) with \(\lambda\in P\) and \(\gamma = (\gamma^\pm_{i,\pm m})_{i \in I, m \in \mathbb{N}}\) such that \(\gamma_{i,\pm m}^\pm \in \mathbb{C}\) and \(\gamma_{i,0}^\pm = q^{\pm \lambda(\alpha_i^\vee)}\) is called an \(l\)-weight. We denote by \(P_l\) the set of \(l\)-weights, and we view it as a refinement of the notion of a weight. \par
A \(\uqhat{g}\)-module \(V\) is said to be of \(l\)-highest weight \((\lambda,\gamma) \in P_l\) if there exists \(v \in V\) such that
\begin{gather*}
	\uqhat{g}^+(v) = 0, \qquad
	\uqhat{g}(v) = V, \qquad
	\phi_{i,\pm m}^\pm (v) = \gamma_{i,\pm m }^\pm v, \qquad
	k_h(v) = q^{\lambda(h)}v.
\end{gather*}
For an \(l\)-weight \((\lambda,\gamma) \in P_l\), we denote by \(L(\lambda,\gamma)\) the unique (up to isomorphism) simple \(\uqhat{g}\)-module of  \(l\)-highest weight \((\lambda,\gamma)\). \par
We let \(P_l^+\) bet the set of dominant \(l\)-weights, which are the \(l\)-weights \((\lambda,\gamma)\) for which there exist polynomials \((P_i)_{i\in I}\) with \(P_i \in 1+u\mathbb{C}[u]\) satisfying
\begin{align*}
	\sum_{m\in\mathbb{N}}\gamma_{i,\pm m}^\pm z^{\pm m} = q^{\deg(P_i)}\frac{P_i(zq^{-1})}{P_i(zq)} \qquad \text{in \(\mathbb{C}[[z^{\pm 1}]]\)}.
\end{align*}
We call the \((P_i)_{i\in I}\) the Drinfeld polynomials of \((\lambda, \gamma)\). We notice that if an \(l\)-weight \((\lambda,\gamma)\) belongs to \(P_l^+\), then \(\lambda \in P^+\).
\begin{theorem}[Theorem 4.9 and Proposition 5.2 in \cite{H05}]
	The simple objects of \(\Oint(\uqhat{g})\) are the modules \(L(\lambda,\gamma)\) with \((\lambda,\gamma) \in P_l^+\). \qed 
\end{theorem}

For a module \(V\) in \(\Oint(\uqhat{g})\), we have a decomposition \(V = \bigoplus_{(\lambda,\gamma) \in P_l} V_{(\lambda,\gamma)}\) with \(V_{(\lambda,\gamma)}\) the subspace of \(V_{\gamma}\) on which \(\left(\phi_{i,\pm m}^{\pm} - \gamma_{i,\pm m}^{\pm}\right)\) acts nilpotently for all \(i\in I\), \(m\in \mathbb{N}\).

\begin{proposition}[Proposition 5.4 in \cite{H05}]\label{proposition-drinfeld}
	Let \(V\) be a module in \(\Oint(\uqhat{g})\) and \((\lambda,\gamma)\) be an \(l\)-weight for which \(V_{(\lambda,\gamma)} \neq 0\). Then there are polynomials \(Q_i,R_i \in 1+u\mathbb{C}[u]\) such that 
	\begin{align*}
		\sum_{m \in \mathbb{N}} \gamma_{i,\pm m}^\pm z^{\pm m} = q^{\deg(Q_i)-\deg(R_i)}\frac{Q_i(zq^{-1})R_i(zq)}{Q_i(zq)R_i(zq^{-1})}  \qquad \text{in \(\mathbb{C}[[z^{\pm 1}]]\)}
	\end{align*} 
	for every \(i \in I\). \qed
\end{proposition}
From this point on, we restrict our attention to types \(A_n\) and \(A^{(1)}_n\). In type \(A_n\), we take \(I = \{1,...,n\}\) as our index set, and in type \(A^{(1)}_n\) we take \(I = \mathbb{Z}/(n+1)\mathbb{Z}\). \par
We consider the formal variables \(Y_{i,a}\) and \(k_\omega\) with \(i \in I\), \(a \in \mathbb{C}^*\) and \(\omega \in \mathfrak{h}^*\). We let \(\Auq{g}\) be the commutative multiplicative group of monomials of the form 
\begin{align*}
	m = k_{\omega(m)} \prod_{i\in I, a \in \mathbb{C}^*} Y_{i,a}^{u_{i,a}(m)}
\end{align*}
for \(\omega(m) \in P\) and \(u_{i,a}(m) \in \mathbb{Z}\) almost all zero, satisfying \(\omega(m)(\alpha_i^\vee) = \sum_{a\in \mathbb{C}^*}u_{i,a}(m)\). The product of two monomials \(m_1, m_2 \in \Auq{g}\) is given by the rule \begin{align*}
\omega(m_1\cdot m_2) = \omega(m_1) + \omega(m_2), && u_{i,a}(m_1\cdot m_2) = u_{i,a}(m_1)+u_{i,a}(m_2).
\end{align*} 
\begin{remark}\label{remark-qchar-A}
	In the \(A_n\) case, the weight \(\omega(m)\) is completely determined by \((u_{i,a}(m))_{i\in I, a\in \mathbb{C}^*}\). We may thus omit the variables \(k_h\) from our notation without losing information. We will still use the notation \(\omega(m)\) whenever we need to refer to the weight of a monomial \(m\).
\end{remark}
For a module \(V\) in \(\Oint(\uqhat{g})\) and an \(l\)-weight \((\lambda,\gamma)\) such that \(V_{(\lambda,\gamma)} \neq 0\), we write \(Q_i(u) = \prod_{a\in\mathbb{C}^*} (1-ua)^{\mu_{i,a}}\) and \(R_i(u) = \prod_{a\in\mathbb{C}^*} (1-ua)^{\nu_{i,a}}\) for the polynomials of \Cref{proposition-drinfeld}, and set
\begin{align*}
	m_{(\lambda,\gamma)} = k_\lambda\prod_{i\in I,a\in\mathbb{C}^*}Y_{i,a}^{\mu_{i,a}-\nu_{i,a}} \in \Auq{g}. 
\end{align*}
We see that a monomial \(m \in \Auq{g}\) satisfies \(u_{i,a}(m) \geq 0\) if and only if \(m = m_{(\lambda,\gamma)}\) for (a unique) \((\lambda,\gamma) \in P_l^+\). In this case we say that \(m\) is a dominant monomial. Moreover, if \(m\) is dominant and \(m=m_{(\lambda,\gamma)}\) with \((\lambda,\gamma)\in P_l^+\), we write \(L(m)\) for \(L(\lambda,\gamma)\). \par
We let \(\Yuq{g}\) be the ring defined as follows:
\begin{enumerate}
	\item its elements are formal infinite sums of monomials in \(\Auq{g}\) with integer coefficients,
	\item for every such formal infinite sum \(\chi\), there are finitely many weights \(\lambda_1,...,\lambda_s \in \mathfrak{h}^*\) such that
	\begin{align*}
		\text{\(m\) is a monomial of \(\chi\)} \implies \omega(m) \in \bigcup_{j=1}^{s} (\lambda_j - Q^+). 
	\end{align*}
\end{enumerate}
\begin{definition}
	The \(q\)-character of a module \(V \in \Oint(\uqhat{g})\) is defined to be
	\begin{align*}
		\chi_q(V) = \sum_{(\lambda,\gamma) \in P_l} \dim V_{(\lambda,\gamma)} m_{(\lambda,\gamma)} \in \Yuq{g}.
	\end{align*}
\end{definition}
Let \(\mathcal{R}(\uqhat{g})\) be the Grothendieck group of finite length modules in \(\Oint(\uqhat{g})\).
\begin{theorem}[Theorem 5.15 and Theorem 6.2 in \cite{H05}]
	The \(q\)-characters give rise to an injective map \(\chi_q: \mathcal{R}(\uqhat{g}) \rightarrow \Yuq{g}\), and its image is a subring of \(\Yuq{g}\). \par
	Moreover, given two modules \(V_1,V_2 \in \Oint(\uqhat{g})\), there exists a module \(W \in \Oint(\uqhat{g})\) such that \(\chi_q(V_1)\cdot \chi_q(V_2) = \chi_q(W)\). \qed
\end{theorem}
\begin{remark}\label{remark-fusion}
	The algebra \(\uqhat{g}\) is not, a priori, a Hopf algebra. The theorem above can be viewed as a consequence of the existence of the fusion product introduced in \cite{H05} 
\end{remark}
For \(i \in I\) and \(a \in \mathbb{C}^*\) we set 
\begin{align*}
	A_{i,a} \coloneq k_{\alpha_i} Y_{i,aq^{-1}}Y_{i,aq}Y^{-1}_{{i-1},a}Y^{-1}_{{i+1},a} \in \Yuq{g}.
\end{align*} 
In the \(A_n\) case, with \(I=\{1,...,n\}\), we have \(Y_{0,a} = 1\) and \(Y_{n+1,a} = 1\), whereas in the \(A^{(1)}_n\) case, with \(I = \mathbb{Z}/(n+1)\mathbb{Z}\), we have \(Y_{0,a} = Y_{n+1,a}\).
We define a partial order on \(\Auq{g}\) by 
\begin{align*}
	m' \leq m \iff m'm^{-1} \in \mathbb{Z}[A_{i,a}^{-1} \mid i \in I, a \in \mathbb{C}^*].
\end{align*}
We call it the Nakajima partial order. 
\begin{theorem}[Theorem 3.2 in \cite{H05-qt}]\label{Htheorem-char-order}
	Let \(m \in \Auq{g}\) be a dominant monomial and let \(m'\) be a monomial that appears in \(\chi_q(L(m))\). Then \(m' \leq m\).
\end{theorem}

\begin{remark}
	In this paper we work with the algebras \(\uqtor{2n}\). For \(n > 1\), they are given by \(\uqtor{2n} \coloneq \uq{\widehat{\widehat{\Sl{2n}}}}\). For \(n=1\), however, we use a modified definition which can be found in \cite{H-folding}. In short, the algebra \(\uqtor{2}\) has a presentation by generators and relations similar to \Cref{qaffinization-def}, but the relation
	\begin{align*}
		x_{i,r+1}^\pm x_{j,r'}^\pm - q^{\pm C_{i,j}} x_{j,r'}^\pm x_{i,r+1}^\pm = q^{\pm C_{i,j}} x_{i,r}^\pm x_{j,r'+1}^\pm - x_{j,r'+1}^\pm x_{i,r}^\pm
	\end{align*}
	is not imposed. \par
	As the results of this paper hold uniformly for \(\uqtor{2n}\) for all \(n \geq 1\), we will not make any distinction between the case \(n = 1\) and the cases \(n>1\). 
\end{remark}

\subsection{Recollection on \(\mathcal{C}_1(\uqsl{n+1})\)}\label{Recollection-C1}

We summarize some facts about the category \(\mathcal{C}_1(\uqsl{n+1})\) that will be of use later on in this paper. We follow \cite{HL10}. \par 
We let \(I_n = \{1,...,n\}\), and let \((I_n)_0\) (respectively \((I_n)_1\)) be the subset of \(I_n\) consisting of even (respectively odd) integers. We let \(\xi : I_\infty \rightarrow \{0,1\}\) be the map defined by 
\begin{align*}
	\xi(i) = \begin{cases}
		0 & i \in (I_n)_0 \\
		1 & i \in (I_n)_1
	\end{cases}.
\end{align*}
We denote by \(\mathcal{M}_n \subseteq \Auq{\Sl{n+1}}\) the subgroup of monomials generated by \(Y_{i,q^{\xi(i)}}\) and \(Y_{i,q^{\xi(i)+2}}\) for all \(i\in I_n\). We let \(\mathcal{C}_1(\uqsl{n+1})\) be the full subcategory of \(\Oint(\uqsl{n+1})\) consisting of modules of finite length, with composition factors isomorphic to \(L(m)\) for some \(m \in \mathcal{M}_\infty\). \par

\begin{proposition}[Proposition 3.2 of \cite{HL10}]\label{HL-grothendieck-polynomial}
	The category \(\mathcal{C}_1(\uqsl{n+1})\) is monoidal. The polynomial ring
	\begin{align*}
		\mathbb{Z}\left[[L(Y_{i,q^{\xi(i)}})],[L(Y_{i,q^{\xi(i)+2}})] \mid i \in I_n \right]
	\end{align*} 
	is isomorphic to the Grothendieck group \(\mathcal{R}(\mathcal{C}_1(\uqsl{n+1}))\). 	\qed
\end{proposition}

We let \(\Gamma_n\) be the following quiver (see Section 4.2 of \cite{HL10}). Its vertex set is \(\{1,...,n\}\cup\{1',...,n'\}\). We think of the vertices \(i\) and \(i'\) as elements of \(I_n\). The arrows of \(\Gamma_n\) are given by
\begin{align*}
	&\Arr_{{\Gamma_n}}(i,j) = \begin{cases}
		\delta_{i,i+1}+\delta_{i,i-1} & i\in (I_n)_0 \\ 0 & i \in (I_n)_1
	\end{cases}, 
	&& \Arr_{{\Gamma_n}}(i',j') = 0, \\
	&\Arr_{{\Gamma_n}}(i',j) = \begin{cases}
		\delta_{i,j} & i\in (I_n)_0 \\ 0 & i \in (I_n)_1
	\end{cases},
	&& \Arr_{{\Gamma_n}}(i,j') = \begin{cases}
		0 & i\in (I_n)_0 \\ \delta_{i,j} & i \in (I_n)_1
	\end{cases}.
\end{align*}
The vertices \(i'\) are considered frozen.

\begin{figure}[h]
	\begin{center}
		\[\begin{tikzcd}
			\boxed{1'} & \boxed{2'} & \boxed{3'} & \boxed{4'} & \boxed{5'} \\
			1 & 2 & 3 & 4 & 5 
			\arrow[from=2-1, to=1-1]
			\arrow[from=2-2, to=2-1]
			\arrow[from=2-2, to=2-3]
			\arrow[from=2-3, to=1-3]
			\arrow[from=2-4, to=2-3]
			\arrow[from=2-4, to=2-5]
			\arrow[from=2-5, to=1-5]
			\arrow[from=1-2, to=2-2]
			\arrow[from=1-4, to=2-4]
		\end{tikzcd}\]
	\end{center}
	\caption{The quiver \(\Gamma_n\) with \(n = 5\). The square vertices are frozen}
\end{figure}
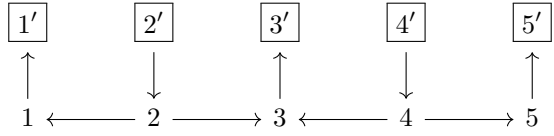 

We set \(\mathcal{F}_n = \mathbb{Z}[f_i, x_i^{\pm 1} \mid i \in I_n]\). We let \(\mathcal{A}_n \subseteq \mathcal{F}_n\) be the cluster algebra defined by \(\Gamma_n\), with initial cluster \((f_i,x_i)_{i\in I_n}\), with the rule that \(x_i\) corresponds to the mutable vertex \(i\) and \(f_i\) to the frozen vertex \(i'\). \par
For an almost positive root \(\alpha \in \Delta\) (of the root system of type \(A_n\)), let \(x[\alpha] \in \mathcal{A}_n\) be the associated (non-frozen) cluster variable. In particular \(x[-\alpha_i] = x_i\). \par
\begin{proposition}[Lemma 4.4 of \cite{HL10}]\label{HL-cluster-polynomial}
	The cluster algebra \(\mathcal{A}_n\) is isomorphic to
	\begin{align*}
		\mathbb{Z}\left[x[-\alpha_i], x[\alpha_i] \mid i \in I_n \right],
	\end{align*}
	a polynomial algebra in \(2n\) variables. \qed
\end{proposition}
To an almost positive root \(\alpha\in \Delta\) we associate a monomial \(m_\alpha \in \mathcal{M}_n\) as follows. We set
\begin{align*}
	m_{-\alpha_i} = \begin{cases}
		Y_{i,q^{\xi(i)+2}} & i \in (I_\infty)_0, \\
		Y_{i,q^{\xi(i)}} & i \in (I_\infty)_1,
	\end{cases} && \text{and} && m_{\alpha_i} = \begin{cases}
		Y_{i,q^{\xi(i)}} & i \in (I_\infty)_0, \\
		Y_{i,q^{\xi(i)+2}} & i \in (I_\infty)_1,
	\end{cases}
\end{align*}
for the simple roots. For a non-simple positive root \(\alpha = \sum_{i \in [i_1,i_2]} \alpha_i\) we set \(m_{\alpha} = \prod_{i\in[i_1,i_2]} m_{\alpha_i}\). \par
We let 
\begin{align*}
	\iota : \mathcal{A}_n \rightarrow \mathcal{R}(\mathcal{C}_1(\uqsl{n+1}))
\end{align*} 
be the ring homomorphism defined by \(\iota(x[\pm \alpha_i]) = [L(m_{\pm\alpha_i})]\). 
\begin{theorem}[Conjecture 4.6 and Section 10 of \cite{HL10}]\label{HR-sln-cluster}
	The map \(\iota\) is an isomorphism, and it satisfies
	\begin{align*}
		\iota(x[\alpha]) = [L(m_\alpha)] && \text{and} && \iota(f_i) = [L(Y_{i,q^{\xi(i)}}Y_{i,q^{\xi(i)+2}})]
	\end{align*}
	for any almost positive root \(\alpha\) and \(i \in I_n\). \qed
\end{theorem}
In this setting, the simple modules in \(\mathcal{C}_1(\uqsl{n+1})\) are real, and their classes are in bijection with the cluster monomials.

\section{The quantum group \(\uqslinf\)}\label{uqslinf}
{In this section we follow \cite{H-folding}, where full details and proofs can be found. \par
	
	We let \(A_\infty\) be the Dynkin diagram with vertex set \(I_\infty \coloneq \mathbb{Z}\) and an edge between \(i\) and \(j\) for \(\lvert j-i \rvert = 1\). We view it as a limit of the Dynkin diagrams of type \(A_n\) when \(n\rightarrow \infty\). 
	
	\begin{figure}[h]
		\begin{center}
			\[\begin{tikzcd}
				\cdots & -2 & -1 & 0 & 1 & 2 & 3 & 4 & 5 & 6 & \cdots
				\arrow[from=1-2, to=1-1, no head]
				\arrow[from=1-3, to=1-2, no head]
				\arrow[from=1-3, to=1-4, no head]
				\arrow[from=1-4, to=1-5, no head]
				\arrow[from=1-5, to=1-6, no head]
				\arrow[from=1-7, to=1-6, no head]
				\arrow[from=1-8, to=1-7, no head]
				\arrow[from=1-9, to=1-8, no head]
				\arrow[from=1-9, to=1-10, no head]
				\arrow[from=1-10, to=1-11, no head]
			\end{tikzcd}\]
		\end{center}
		\caption{The Dynkin diagram \(A_\infty\)}
	\end{figure}
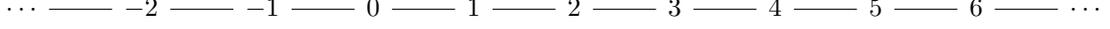 
	We let \(\mathfrak{h}\) and \(\mathfrak{h}^*\) be the associated spaces, \(\Delta \subseteq \mathfrak{h}^*\) be the root system, \(\{\alpha_i \mid i \in I_\infty\} \subseteq \Delta\) the set of simple roots, and set \(Q^+ = \sum_{i\in I_\infty} \mathbb{N}\alpha_i\). We also let \(P \subseteq \mathfrak{h}^*\) be the group of integral weights and \(P^+\) be its submonoid of dominant integral weights. We use \Cref{qaffinization-def} to associate to this Dynkin diagram an algebra \(\uqslinf\) with \(q \in \mathbb{C}^*\), not a root of unity. We think of it as a double-infinite analogue of the quantum affine algebras \(\uqsl{n}\). 
	We let \(\Oint(\uqslinf)\) be defined analogously to \Cref{Oint-def}. See also Definition 2.1 and Section 3.2 of \cite{H-folding}. We let
	\begin{align*}
		P_l = \{\gamma = (\gamma_{i,\pm m}^\pm)_{i\in I_\infty, m \in \mathbb{N}} \mid \gamma_{i,\pm m}^{\pm} \in \mathbb{C}\}
	\end{align*}
	be the set of \(l\)-weight. \par
	We also define \(P_l^+\), the set of dominant \(l\)-weights. These are the \(l\)-weights \(\gamma \in P_l\) for which there exist Drinfeld polynomials \((P_i)_{i\in I}\) such that \(P_i \in 1+u\mathbb{C}[u]\) and \(P_i = 1\) for almost all \(i \in I_\infty\), satisfying
	\begin{align*}
		\sum_{m\in\mathbb{N}}\gamma_{i,\pm m}^\pm z^{\pm m} = q^{\deg(P_i)}\frac{P_i(zq^{-1})}{P_i(zq)} \qquad \text{in \(\mathbb{C}[[z^{\pm 1}]]\)}. 
	\end{align*}
	
	For every \(n \in \mathbb{N}\), the algebra \(\uqslinf\) contains a subalgebra \(\hat{\mathcal{U}}_{[-n,n]}\), isomorphic to \(\uqsl{2n+2}\), such that \(\hat{\mathcal{U}}_{[-n,n]} \subseteq \hat{\mathcal{U}}_{[-(n+1),n+1]}\), and 
	\begin{align*}
		\uqslinf = \bigcup_{n\in\mathbb{N}}\hat{\mathcal{U}}_{[-n,n]}.
	\end{align*}
	\begin{theorem}[Theorem 3.8 in \cite{H-folding}]\label{htheorem-simples} The simple modules in \(\Oint(\uqslinf)\) are the simple \(l\)-highest weight modules \(L(\gamma)\) for \(\gamma \in P_l^+\). Moreover,
		\begin{align*}
			L(\gamma) = \bigcup_{n\in\mathbb{N}} L_n(\gamma),
		\end{align*}
		where \(L_n(\gamma) = \hat{\mathcal{U}}_{[-n,n]}v_\gamma\) and \(v_\gamma\) is an \(l\)-highest weight vector of \(L(\gamma)\). For every \(n \in \mathbb{N}\), the \(\uqsl{2n+2}\)-module \(L_n(\gamma)\) is the simple, finite-dimensional module associated to the Drinfeld polynomials \((P_i)_{i \in [-n,n]}\). \qed
	\end{theorem}
	
	\begin{proposition}[Proposition 3.11 and Corollary 3.12 in \cite{H-folding}]\label{Hproposition-drinfeld}
		Let \(V\) be a module in \(\Oint(\uqslinf)\) and let \(\gamma\) be an \(l\)-weight such that \(V_\gamma \neq 0\). For every \(i \in I_\infty\), there are polynomials \(Q_i,R_i \in 1+u\mathbb{C}[u]\) such that 
		\begin{align*}
			\sum_{m \in \mathbb{N}} \gamma_{i,\pm m}^\pm z^{\pm m} = q^{\deg(Q_i)-\deg(R_i)}\frac{Q_i(zq^{-1})R_i(zq)}{Q_i(zq)R_i(zq^{-1})}  \qquad \text{in \(\mathbb{C}[[z^{\pm 1}]]\)}
		\end{align*}
		with \(Q_i = 1\), \(R_i = 1\) for almost all \(i \in I_\infty\). \qed
		
	\end{proposition}
	
	We consider the formal variables \(Y_{i,a}\) with \(i \in I_\infty\) and \(a \in \mathbb{C}^*\). We let \(\Asub{\infty}\) and \(\Ysub{\infty}\) be defined as in \Cref{Section-quantum-groups} (see also \Cref{remark-qchar-A}). \par 
	For a \(V\) in \(\Oint(\uqslinf)\) and an \(l\)-weight \(\gamma\) such that \(V_\gamma\), we write \(Q_i(u) = \prod_{a\in\mathbb{C}^*} (1-ua)^{\mu_{i,a}}\) and \(R_i(u) = \prod_{a\in\mathbb{C}^*} (1-ua)^{\nu_{i,a}}\) for the polynomials of \Cref{Hproposition-drinfeld}. We associate to \(\gamma\) the monomial
	\begin{align*}
		m_\gamma = \prod_{i\in I_\infty, \  a\in\mathbb{C}^*}Y_{i,a}^{\mu_{i,a}-\nu_{i,a}} \in \Asub{\infty},
	\end{align*}
	and define the \(q\)-character of \(V\) to be \(\chi_q(V) = \sum_{\gamma}\dim(V_\gamma)m_\gamma\). We identify the \(q\)-characters of \(\hat{\mathcal{U}}_{[-n,n]}\)-modules in \(\Oint(\uqsl{2n+1})\) as elements of \(\Ysub{\infty}\) in an obvious way. \par 
	For an element \(\chi = \sum_{j}c_j m_j \in \Ysub{\infty}\) with \(c_j \in \mathbb{Z}\) and \(m_j \in \Asub{\infty}\), and a weight \(\lambda \in P\), we set 
	\begin{align*}
		\chi_\lambda = \sum_{j,\  \omega(m_j)=\lambda} c_j m_j.
	\end{align*}
	\begin{proposition}[Proposition 3.11 in \cite{H-folding}]\label{hproposition-char}
		\emergencystretch 3em {Let \(L(\gamma)\) be a simple \(\uqslinf\)-module in \(\Oint(\uqslinf)\), and let \(L_n(\gamma)\) be the simple \(\uqsl{2n+2}\)-modules of \Cref{htheorem-simples}. Then for any weight \(\lambda \in P\),
		\begin{align*}
			\chi_q(L(\gamma))_\lambda = \chi_q(L_n(\gamma))_\lambda
		\end{align*}
		for \(n\) large enough.} \qed
	\end{proposition}
	
	\begin{remark}\label{remark-char-product}
		Let \(L(\gamma_1)\) and \(L(\gamma_2)\) be two simple \(\uqslinf\)-modules in \(\Oint(\uqslinf)\). Then 
		\begin{align*}
			(\chi_q(L(\gamma_1))\chi_q(L(\gamma_2)))_\lambda = (\chi_q(L_n(\gamma_1))\chi_q(L_n(\gamma_2)))_\lambda 
		\end{align*}
		for \(n\) large enough. In particular, we can use known multiplication formulas for \(\uqsl{n}\) to deduce analogous formulas for \(\uqslinf\). 
	\end{remark}
	
	For every dominant monomial \(m \in \Asub{\infty}\) there exists a unique \(\gamma \in P_l^+\) such that \(m = m_{\gamma}\); we will write \(L(m) = L(\gamma)\). For \(i\in I_\infty\), \(a\in \mathbb{C}^*\), we let \(A_{i,a} \in \Asub{\infty}\) be the monomial \(Y_{i,aq^{-1}}Y_{i,aq}Y_{{i-1},a}^{-1}Y_{{i+1},a}^{-1}\). We define a partial order on \(\Asub{\infty}\)  by
	\begin{align*}
		m \leq m' \iff m'm^{-1} \in \mathbb{Z}[A_{i,a}^{-1} \mid i \in I_\infty, a\in \mathbb{C}^*].
	\end{align*}
	
	\begin{corollary}\label{corollary-order}
		Let \(m \in \Asub{\infty}\) be a dominant monomial and let \(m'\) be a monomial which appears in \(\chi_q(L(m))\). Then \(m' \leq m\).
	\end{corollary}
	\begin{proof}
		This follows from \Cref{hproposition-char} and \Cref{Htheorem-char-order}.
	\end{proof}
	
}

\section{Folding of quivers and cluster algebras}\label{folding}
In this section, we will discuss a certain notion of folding of locally finite infinite quivers and associated cluster algebras. Because we will work with infinite quivers and their mutations, we need to discuss the folding in a relatively general setting. Most of this section is inspired by \cite{FWZ21}. \par

\begin{definition}\label{G-admissible}
	Let \(Q\) be a locally finite quiver and let \(G\subseteq \Sym(Q_0)\) be a group acting on the vertex set \(Q_0\). We say that \(Q\) is \emph{strongly \(G\)-admissible} if the following hold
	\begin{enumerate}
		\item\label{G-inv} \(\Arr_Q(v_1,v_2) = \Arr_Q(gv_1,gv_2)\) for any \(v_1,v_2 \in Q_0\), \(g\in G\),
		\item\label{stab} the action of \(G\) on \(Q_0\) is free, that is, the stabilizer of any vertex \(v\in Q_0\) is trivial, 
		\item\label{virt-loop} \(\Arr_Q(Gv,Gv) = 0\) for any \(v \in Q_0\),
		\item\label{virt-cyc} \(\Arr_Q(Gv_1,v_2)\cdot\Arr_Q(v_2,Gv_1) = 0\) for any \(v_1,v_2 \in Q_0\). 
	\end{enumerate}
	If, moreover \(Q\) has a frozen part \(F\), then the ice quiver \((Q,F)\) is said to be \emph{strongly \(G\)-admissible} if \(Q\) is strongly \(G\)-admissible and \(G(F) = F\). In that case, each \(G\)-orbit is either mutable or frozen.
\end{definition}

\begin{definition}
	Let \(Q\) be a quiver and \(G\) a group acting on \(Q_0\). Suppose \(Q\) is \(G\)-admissible. We define a new quiver \(Q^G\) with set of vertices \(Q_0/G\) and arrows \(\Arr_{Q^G}(Gv_1,Gv_2) = \Arr_Q(Gv_1,v_2)\) for any \(v_1,v_2 \in Q_0\). We say \(Q^G\) is \emph{the folding of \(Q\) with respect to \(G\)}. \par
	If \(Q\) has a frozen part \(F\), we let \(F^G\) be the set of frozen vertices \(F/G\) of \(Q^G\). We say \((Q^G,F^G)\) is \emph{the folding of \((Q,F)\) with respect to \(G\)}. 
\end{definition}

\begin{remark}
	Conditions \ref{G-inv} and \ref{stab} imply \(\Arr_Q(Gv_1,v_2) = \Arr_Q(v_1,Gv_2)\) for any \(v_1,v_2 \in Q_0\). Indeed,
	\begin{align*}
		\Arr_Q(Gv_1,v_2) & = \sum_{v \in Gv_1}\Arr(v,v_2) = \sum_{g \in G}\Arr(gv_1,v_2) \\ & = \sum_{g \in G}\Arr(v_1,g^{-1}v_2) = \sum_{v' \in Gv_2}\Arr(v_1,v') = \Arr_Q(v_1,Gv_2).
	\end{align*}
\end{remark}

\begin{remark}
	Our definition is different than that of \cite{FWZ21}. Here we allow the quiver \(Q\) to have infinitely many vertices, but we require that it folds onto a (non-valued) quiver (condition \ref{stab}). There is no evident way to fold an infinite quiver onto a (maybe valued) quiver without imposing some symmetry conditions. We use the word ``strongly" to distinguish between our notion of admissibility and the one in \cite{FWZ21}.
\end{remark}
\begin{remark}
	Conditions \ref{virt-loop} and \ref{virt-cyc} ensure that the quiver \(Q\) has no ``virtual" loops or \(2\)-cycles. This means that there are no arrows nor oriented 2-paths between two vertices belonging to the same \(G\)-orbit. In particular, this ensures that the folded quiver has no loops or \(2\)-cycles.
\end{remark}

\begin{lemma}
	Let \(Q\) be a quiver and \(G\subseteq \Sym(Q_0)\), such that \(Q\) is strongly \(G\)-admissible. Then the following hold.
	\begin{enumerate}
		\item Let \(v_1\) and \(v_2\) be two vertices in the same \(G\)-orbit, then \(\mu_{v_1}\mu_{v_2}Q = \mu_{v_2}\mu_{v_1}Q\). 
		\item Let \(w_1\) and \(w_2\) be two vertices in \(Q\), then \(\Arr_Q(w_1,Q_0)\cdot \Arr_Q(Q_0,w_2) < \infty\).
		\item Let \(w_1\), \(w_2\) and \(v\) be vertices in \(Q\), then \(\Arr_{\mu_{gv}Q}(w_1,w_2) = \Arr_{Q}(w_1,w_2)\) for almost all \(g \in G\). 
	\end{enumerate}
\end{lemma}
\begin{proof}
	\begin{enumerate}
		\item This follows immediately from condition \ref{virt-loop} of \Cref{G-admissible}. See also \cite{FWZ21}. 
		\item The inequality \(\Arr_Q(w_1,Q_0)\cdot \Arr_Q(Q_0,w_2) < \infty\) is an immediate consequence of the local finiteness of \(Q\). 
		\item The inequality \(\Arr_Q(w_1,Q_0)\cdot \Arr_Q(Q_0,w_2) < \infty\) implies \(\Arr_Q(w_1,gv)\cdot \Arr_Q(gv,w_2) = 0\) for almost all \(g \in G\). The equality \(\Arr_{\mu_{gv}Q}(w_1,w_2) = \Arr_{Q}(w_1,w_2)\) for almost all \(g \in G\) is a consequence of the mutation rule. \qedhere
	\end{enumerate}
\end{proof}

\begin{definition}
	Let \(Q\) be a quiver and \(G\subseteq \Sym(Q_0)\), such that \(Q\) is strongly \(G\)-admissible. Given a \(G\)-orbit \(K\), we define \(\hmu_K Q\), \emph{the orbit-mutation of \(Q\) at \(K\)} as the quiver obtained by applying mutation at all vertices of \(K\). In view of the lemma above, this operation is well defined.
\end{definition}

\begin{lemma}\label{mutation-behavior}
	Let \(Q\) be a quiver and \(G\subseteq \Sym(Q_0)\). Suppose \(Q\) is strongly \(G\)-admissible and let \(K\) be a \(G\)-orbit. Then the following statements hold.
	\begin{enumerate}
		\item The quiver \(\hmu_K Q\) satisfies conditions \ref{G-inv}, \ref{stab} and \ref{virt-loop} of \Cref{G-admissible} with respect to \(G\). 
		\item \(\Arr_{\mu_K (Q^G)} (Gv_1,Gv_2) \leq \Arr_{\hmu_K Q}(Gv_1,v_2)\) for every \(v_1,v_2 \in Q_0\).
		\item \(\hmu_K Q\) satisfies condition \ref{virt-cyc} of \Cref{G-admissible} if and only if \begin{align*}
			\Arr_{\hmu_K Q}(Gv_1,v_2) = \Arr_{\mu_K (Q^G)} (Gv_1,Gv_2)
		\end{align*} 
		for every \(v_1,v_2 \in Q_0\). In this case \(\mu_K(Q^G) = (\hmu_K Q)^G\).
	\end{enumerate}
\end{lemma}
\begin{proof} \begin{enumerate}
	\item Condition \ref{stab} holds for \(\hmu_K Q \) trivially. We can calculate the arrows in the orbit-mutated quiver using the following formula
	\begin{align*}
		\Arr_{\hmu_KQ}(v_1,v_2) = \begin{cases}
			\Arr_{Q}(v_2,v_1) & \text{if \(v_1 \in K\) or \(v_2 \in K\)} \\
			\max\{0, a_{v_1,v_2} -a_{v_2,v_1}\} & \text{otherwise}
		\end{cases},
	\end{align*}
	where \(a_{w_1,w_2} = \Arr_{Q}(w_1,w_2) + \sum_{k \in K} \Arr_{Q}(w_1,k)\cdot \Arr_{Q}(k,w_2)\) for \(w_1,w_2\in Q_0\). Conditions \ref{G-inv} and \ref{virt-loop} then follow immediately.  \par
	\item If \(v_1\) or \(v_2\) belong to \(K\), it is clear that \(\Arr_{\hmu_K Q}(Gv_1,v_2) =  \Arr_{\mu_K (Q^G)} (Gv_1,Gv_2)\).
	So suppose \(v_1,v_2 \notin K\). If \(\Arr_{\mu_K(Q^G)}(Gv_1,Gv_2) = 0\) the statement is trivial. Otherwise, following the definition of mutation, we must have
	\begin{align*}
		\Arr_{\mu_K(Q^G)}(Gv_1,Gv_2) =\ & \Arr_{Q}(Gv_1,v_2) - \Arr_{Q}(v_2,Gv_1) \\ 
		& + \Arr_{Q}(v_1,K)\Arr_{Q}(K,v_2) - \Arr_{Q}(v_2,K)\Arr_{Q}(K,v_1) \\
		=\ & \Arr_{Q}(Gv_1,v_2) - \Arr_{Q}(v_2,Gv_1) \\ 
		& + \sum_{k' \in K}\sum_{k \in K} \left(\Arr_{Q}(v_1,k')\Arr_{Q}(k,v_2)-\Arr_{Q}(v_2,k)\Arr_{Q}(k',v_1)\right) \\
		=\ & \Arr_{Q}(Gv_1,v_2) - \Arr_{Q}(v_2,Gv_1) \\
		& + \sum_{g\in G}\sum_{k \in K} \left(\Arr_{Q}(gv_1,k)\Arr_{Q}(k,v_2)-\Arr_{Q}(v_2,k)\Arr_{Q}(k,gv_1)\right) \\
		=\ & \Arr_{\hmu_KQ}(Gv_1,v_2) - \Arr_{\hmu_KQ}(v_2,Gv_1) \\
		\leq\ & \Arr_{\hmu_KQ}(Gv_1,v_2). 
	\end{align*}
	This proves the desired inequality. \par
	\item Following the previous part, we only need to show the if and only if statement for vertices \(v_1,v_2 \notin K\). \par
	If \(\Arr_{\hmu_KQ}(Gv_1,v_2)\Arr_{\hmu_KQ}(v_2,Gv_1) = 0\), then without loss of generality \(\Arr_{\hmu_KQ}(Gv_1,v_2) = 0\), which implies \(0 = \Arr_{\hmu_KQ}(Gv_1,v_2) \geq \Arr_{\mu_K (Q^G)} (Gv_1,Gv_2)\), in particular 
	\begin{align*}
		\Arr_{\hmu_KQ}(Gv_1,v_2) = \Arr_{\mu_K (Q^G)} (Gv_1,Gv_2).
	\end{align*}
	From the previous part we then deduce that 
	\begin{align*}
		\Arr_{\mu_K (Q^G)} (Gv_2,Gv_1) = \Arr_{\hmu_KQ}(Gv_2,v_1) - \Arr_{\hmu_KQ}(Gv_1,v_2) = \Arr_{\hmu_KQ}(Gv_2,v_1).
	\end{align*} 
	In the other direction, we have 
	\begin{align*}
		\Arr_{\hmu_KQ}(Gv_1,v_2)\Arr_{\hmu_KQ}(v_2,Gv_1) = \Arr_{\mu_K(Q^G)}(Gv_1,Gv_2)\Arr_{\mu_K(Q^G)}(Gv_2,Gv_1) = 0.
	\end{align*}
	When the above holds, the equality \(\mu_K(Q^G) = (\hmu_KQ)^G\) follows immediately. \qedhere
	\end{enumerate}
\end{proof}

\begin{definition}\label{strongly-gf}
	Let \(Q\) be a quiver and \(G\subseteq \Sym(Q_0)\) such that \(Q\) is strongly \(G\)-admissible. Then \(Q\) is said to be \emph{strongly globally foldable with respect to \(G\)} if for any finite sequence of orbits \(K_1,...,K_l\), the quiver \((\hmu_{K_i}\circ \cdots \circ \hmu_{K_1})(Q)\) is strongly \(G\)-admissible, for all \(1 \leq i \leq l\). \par
	Let \(F \subseteq Q_0\) be a frozen part such that \((Q,F)\) is strongly \(G\)-admissible. The ice quiver \((Q,F)\) is said to be \emph{strongly globally foldable with respect to \(G\)} if for any finite sequence of mutable orbits \(K_1,...,K_l\), the quiver \((\hmu_{K_i}\circ \cdots \circ \hmu_{K_1})(Q)\) is strongly \(G\)-admissible, for all \(1 \leq i \leq l\).
\end{definition}

\begin{remark}\label{remark-vitrual-2-cycles}
	According to \Cref{mutation-behavior}, for a strongly \(G\)-admissible quiver \(Q\), the quiver \(\hmu_{K} Q\) is strongly \(G\)-admissible if and only if it satisfies condition \ref{virt-cyc} of \Cref{G-admissible}. That is, if and only if we have not created ``virtual" \(2\)-cycles while mutating at an orbit. Therefore, to verify that a quiver \(Q\) is strongly globally foldable with respect to \(G\), it suffices to verify that we cannot create those virtual \(2\)-cycles by process of orbit-mutation. \par
	Moreover, we notice that the equality 
	\begin{align*}
		\Arr_{\mu_K (Q^G)}(Gv_2,Gv_1) - \Arr_{\mu_K (Q^G)}(Gv_1,Gv_2) = \Arr_{\hmu_K (Q)}(Gv_2,v_1) - \Arr_{\hmu_K (Q)}(Gv_1,v_2)
	\end{align*}
	always holds, so
	\begin{align*}
		\Arr_{\mu_K (Q^G)}(Gv_1,Gv_2) = \Arr_{\hmu_K (Q)}(Gv_1,v_2)  \iff \Arr_{\mu_K (Q^G)}(Gv_2,Gv_1) = \Arr_{\hmu_K (Q)}(Gv_2,v_1).
	\end{align*}
	This will be our strategy later to prove that the infinite quivers belonging to a certain class are strongly globally foldable. \par
	
\end{remark}

Let \((Q,F)\) be an ice quiver and \(G \subseteq \Sym(Q_0)\) such that \((Q,F)\) is strongly globally foldable with respect to \(G\). Let \((Q^G,F^G)\) be the folded quiver. We consider the algebras
\begin{align*}
	\mathcal{F} & = \mathbb{Z}[x_v | v \in F][x_v^{\pm1} \mid v \in Q_0\setminus F], \\
	\mathcal{F}^G & = \mathbb{Z}[x_{Gv} | Gv \in F^G][x_{Gv}^{\pm1} \mid Gv \in Q_0^G\setminus F^G].
\end{align*} 
We associate to \((Q,F)\) and \((Q^G,F^G)\) the cluster algebras \(\mathcal{A}\) and \(\mathcal{A}^G\) with initial clusters \((x_v)_{v\in Q_0}\) and \((x_{Gv})_{Gv \in Q_0^G}\), respectively. We consider these cluster algebras as subalgebras of \(\mathcal{F}\) and \(\mathcal{F}^G\), respectively. \par
Any set of cluster variables of \(\mathcal{A}\) obtained from the initial cluster by an iterated sequence of orbit-mutations is called an \emph{orbit-cluster}. In particular, every cluster variable of an orbit-cluster is a cluster variable in the usual sense. \par
We have the obvious homomorphism \(\psi : \mathcal{F} \rightarrow \mathcal{F}^G\) defined by \(\psi(x_v) = x_{Gv}\). Given a cluster variable \(x\) of \(\mathcal{A}\) which belongs to an orbit-cluster, we call \(\psi(x)\) the \emph{folded cluster variable} of \(x\) in \(\mathcal{A}^G\). This terminology is justified by the following proposition.

\begin{proposition}\label{proposition-cluster-folding}
	We keep the notation as above. Let \(K_1,...,K_l\) be a sequence of mutable orbits in \((Q,F)\). The map \(\psi\) maps the orbit-cluster of \(\mathcal{A}\) obtained by the sequence of orbit-mutations at \(K_1,...,K_l\) to the cluster of \(\mathcal{A}^G\) obtained by the sequence of mutations at \(K_1,...,K_l\).
\end{proposition}
\begin{proof}
	By recursion, it suffices to prove the case of a single orbit-mutation. So let \(K\) be a mutable \(G\)-orbit of \((Q,F)\). Orbit-mutating at \(K\), we get the new cluster \((x_v')_{v\in Q_0}\) such that \(x_v' = x_v\) if \(v\notin K\), and 
	\begin{align}\label{orbit-mutated-cluster}
		x_v' = \frac{\prod_{w\in Q_0}x_w^{\Arr_Q(v,w)} + \prod_{w\in Q_0}x_w^{\Arr_Q(w,v)}}{x_v}
	\end{align} 
	if \(v\in K\). Similarly, the cluster \((x_{Gv}')_{Gv \in Q_0^G}\) obtained by mutating \((Q^G,F^G)\) at \(K\) is defined by \(x_{Gv}' = x_{Gv}\) for \(Gv \neq K\) and
	\begin{align}\label{folded-mutated-cluster}
		x_K' = \frac{\prod_{Gw\in Q_0^G}x_{Gw}^{\Arr_{Q^G}(K,Gw)} + \prod_{Gw\in Q_0^G}x_{Gw}^{\Arr_{Q^G}(Gw,K)}}{x_K}.
	\end{align}
	Applying \(\psi\) to \eqref{orbit-mutated-cluster}, we obtain
	\begin{align*}
		\psi(x_v') & = \frac{\prod_{w\in Q_0}x_{Gw}^{\Arr_Q(v,w)} + \prod_{w\in Q_0}x_{Gw}^{\Arr_Q(w,v)}}{x_K} \\ 
		& = \frac{\prod_{Gw\in Q_0^G}x_{Gw}^{\sum_{g \in G}\Arr_Q(v,gw)} + \prod_{Gw\in Q_0^G}x_{Gw}^{\sum_{g \in G}\Arr_Q(gw,v)}}{x_K} \\
		& = \frac{\prod_{Gw\in Q_0^G}x_{Gw}^{\Arr_Q(v,Gw)} + \prod_{Gw\in Q_0^G}x_{Gw}^{\Arr_Q(Gw,v)}}{x_K} \\ 
		& = \frac{\prod_{Gw\in Q_0^G}x_{Gw}^{\Arr_{Q^G}(K,Gw)} + \prod_{Gw\in Q_0^G}x_{Gw}^{\Arr_{Q^G}(Gw,K)}}{x_K}. 
	\end{align*}
	Comparing this with \eqref{folded-mutated-cluster}, the result follows.
\end{proof}

We have the following immediate corollary.

\begin{corollary}
	In the notation above, every cluster variable in \(\mathcal{A}^G\) is the folded cluster variable of some cluster variable in \(\mathcal{A}\). \qedhere
\end{corollary}

\section{Folding the \(A_\infty\) quiver}\label{section_folding_Ainf}

We let \(Q\) be a quiver (without loops or 2-cycles) whose underlying graph is a cycle with \(n\) vertices, \(n\geq 2\), indexed by \(\mathbb{Z}/n\mathbb{Z}\). We let \(A_Q\) be the quiver whose underlying diagram is \(A_\infty\), so its vertices are indexed by \(\mathbb{Z}\). There is an arrow from \(i\) to \(i\pm1\) in \(A_Q\) if and only if there is an arrow from \([i]\) to \([i\pm1]\) in \(Q\), see \Cref{AQ-to-Ainfty}. Let \(\Sigma \in \Sym((A_Q)_0)\) be given by \(\Sigma(i) = i+n\). Let \(G\) be the subgroup of \(\Sym((A_Q)_0)\) generated by \(\Sigma\). 

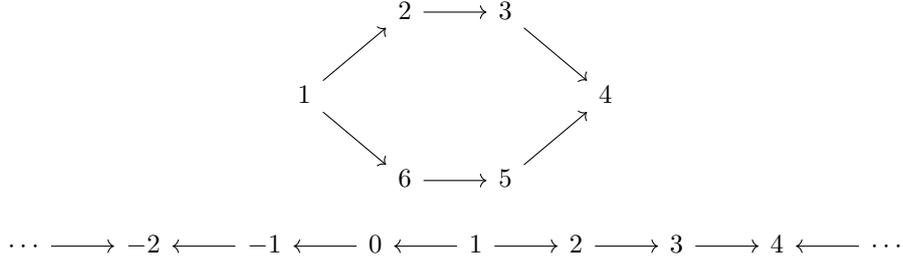
\begin{figure}[h]
	\begin{center}
		\[\begin{tikzcd}
			& 2 & 3 \\
			1 &&& 4 \\
			& 6 & 5
			\arrow[from=1-2, to=1-3]
			\arrow[from=1-3, to=2-4]
			\arrow[from=2-1, to=1-2]
			\arrow[from=2-1, to=3-2]
			\arrow[from=3-2, to=3-3]
			\arrow[from=3-3, to=2-4]
		\end{tikzcd}\]
		\[\begin{tikzcd}
			\cdots & {-2} & {-1} & 0 & 1 & 2 & 3 & 4 &  \cdots
			\arrow[from=1-1, to=1-2]
			\arrow[from=1-3, to=1-2]
			\arrow[from=1-4, to=1-3]
			\arrow[from=1-5, to=1-4]
			\arrow[from=1-5, to=1-6]
			\arrow[from=1-6, to=1-7]
			\arrow[from=1-7, to=1-8]
			\arrow[from=1-9, to=1-8]
		\end{tikzcd}\]
	\end{center}
	\caption{A cyclic quiver with 6 vertices and the associated orientation of \(A_\infty\)}\label{AQ-to-Ainfty}
\end{figure}

\begin{proposition}
	The quiver \(A_Q\) is strongly \(G\)-admissible, and \(Q \simeq A_Q^G\). 
\end{proposition}
\begin{proof}
	The quiver \(A_Q\) and the group \(G\) trivially satisfy the assumptions of \Cref{G-admissible}. The isomorphism \(Q \simeq A_Q^G\) follows from the definition of \(A_Q\) above.
\end{proof}

We wish to understand for what orientations of \(Q\) the quiver \(A_Q\) is strongly globally foldable with respect to \(G\). A partial result will follow from the following general proposition.

\begin{proposition}\label{tree-proposition}
	Let \(A\) be a connected quiver and \(G \subseteq \Sym(A_0)\) a group of permutations of the vertices of \(A\) such that \(A\) is strongly \(G\)-admissible. If \(A^G\) is a finite tree then \(G = \{\mathrm{Id}_A\}\) and \(A = A^G\).
\end{proposition}
\begin{proof}
	Let \(\{O_1,...,O_n\}\) be the \(G\)-orbits in \(A\). If \(G\) is not trivial, then \(\lvert O_1 \rvert = \lvert G \rvert > 1\). Let \(v_1,v_2 \in O_1\) be two distinct vertices of \(A\). \par
	Since \(A\) is connected, we have a (not necessarily directed) path connecting \(v_1\) and \(v_2\). Moreover, we can assume this path does not go through a given vertex more than once. In particular, this means that there is a sequence of distinct vertices \((a_1,...,a_k)\) of \(A\) with \(a_1 = v_1\), \(a_k = v_2\) such that
	\begin{align*}
		\Ed_A(a_1,a_2)\cdot...\cdot\Ed_A(a_{k-1},a_k) > 0. 
	\end{align*} 
	But now we can find \(1 < j \leq k\) such that \(Ga_1 = Ga_j\) and \(Ga_1 \neq Ga_s\) for any \(1 < s < j\). Notice that \(j>2\), otherwise we would have \(\Ed_A(a_1,a_2) > 0\) for \(a_1,a_2\) in the same \(G\)-orbit, which contradicts the strong \(G\)-admissibility of \(A_Q\). \par
	If \(j = 3\), then \(\Ed_A(a_2,O_1) \geq \Ed_A(a_2,a_1) + \Ed_A(a_2,a_3)\geq 2\). So \(\Ed_{A^G}(G{a_{2}},O_1) \geq 2\), which contradicts \(A^G\) being a tree. \par
	If \(j > 3\), we have by construction
	\begin{align*}
		\Ed_A(a_1,a_{2})\cdot...\cdot\Ed_A(a_{j-1},a_j) > 0. 
	\end{align*}
	As \(\Ed_A(a_i,a_{i+1}) \leq \Ed_{A^G}(Ga_i,Ga_{i+1})\), we obtain
	\begin{align*}
		\Ed_{A^G}(O_1,Ga_{2})\cdot...\cdot\Ed_{A^G}(Ga_{j-1},O_1) > 0,
	\end{align*}
	which is again a contradiction to \(A^G\) being a tree.
\end{proof}

\begin{corollary}
	If \(Q\) is cyclically oriented, then \(A_Q\) is not strongly globally foldable with respect to \(G\). 
\end{corollary}
\begin{proof}
	It is well known (see for example \cite{FWZ21}) that \(Q\) is of type \(D_n\), that is, \(Q\) is mutation-equivalent to a quiver \(\tilde{Q}\) whose underlying graph is \(D_n\). \par
	Assume that \(A_Q\) is strongly globally foldable with respect to \(G\), and folds onto \(Q\). Then a finite sequence of orbit-mutations applied to \(A_Q\) yields a quiver \(\tilde{A}\) such that \(\tilde{A}^G = \tilde{Q}\). By the above proposition, \(\tilde{A}\) is finite, which is a contradiction.
\end{proof}

We can now give a precise statement for when \(A_Q\) is strongly globally foldable with respect to \(G\).  

\begin{theorem}\label{folding-theorem}
	The quiver \(A_Q\) is strongly globally foldable with respect to \(G\) if and only if \(Q\) is not cyclically oriented.  
\end{theorem}
\begin{proof}
	From the corollary above, we only need to show that if \(Q\) is not cyclically oriented, then \(A_Q\) is strongly globally foldable with respect to \(G\). \par
	From \cite{FST08}, if \(Q\) is not cyclically oriented, then it is the quiver of a triangulation \(T\) of an annulus with \(k_1\) marks on one boundary component and \(k_2\) marks on the other boundary component such that \(k_1+k_2 = n\). Here \(k_1\) is the number of arrows in \(Q\) pointing clockwise. \par
	Mutations of \(Q\) correspond to flipping the associated arcs in \(T\). We remark that at least one arc in \(T\) connects one boundary component with the other. We enumerate the marks on the outer boundary by \(v_1,...,v_{k_1}\) and the inner boundary by \(u_1,...,u_{k_2}\), both in a clockwise orientation, and such that there is an arc in \(T\) connecting \(v_1\) and \(u_1\). \par
	Let us lift this triangulation to a periodic triangulation of an infinite ribbon \(\mathcal{S}\). To do so, we consider an infinite rectangle with two infinite line boundary components, with an infinite countable set of marked points \(\mathcal{M}\) on each boundary component. We index the marked points on the upper boundary by \(\{\Sigma^i v_j \mid i \in \mathbb{Z}, 1\leq j \leq k_1\}\) and those on the lower boundary \(\{\Sigma^i u_j \mid i \in \mathbb{Z}, 1\leq j \leq k_2\}\), in such a way that \(\Sigma^i v_j\) is to the left of \(\Sigma^{i'} v_{j'}\) if and only if \(i<i'\) or \(i=i'\) and \(j<j'\). We define a shift \(\Sigma : \mathcal{M} \rightarrow \mathcal{M}\) in the obvious way, that is, for \(m \in \mathcal{M}\) we have
	\begin{align*}
		\Sigma(m) = 
		\begin{cases}
			\Sigma^{i+1}v_j & \text{if } m = \Sigma^{i}v_j \\
			\Sigma^{i+1}u_j & \text{if } m = \Sigma^{i}u_j
		\end{cases}.
	\end{align*}
	Arcs in \((\mathcal{S},\mathcal{M})\) are in bijection with non-neighboring pairs of distinct points in \(\mathcal{M}\). We can thus extend the shift \(\Sigma\) to arcs as well; if \(\beta\) is the arc with end points \(m_1,m_2 \in \mathcal{M}\) then \(\Sigma(\beta)\) is the arc with end points \(\Sigma(m_1),\Sigma(m_2)\). \par
	We add an arc \(\gamma_i\) connecting \(\Sigma^i v_1\) with \(\Sigma^i u_1\) for every \(i \in \mathbb{Z}\). We let \(R_i\) be the rectangle bounded by \(\gamma_i\) and \(\gamma_{i+1}\). We triangulate \(R_i\) in such a way that if we identify \(\Sigma^i v_1\) with \(\Sigma^{i+1}v_1\), \(\Sigma^{i}u_1\) with \(\Sigma^{i+1}u_1\) and \(\gamma_i\) with \(\gamma_{i+1}\) we obtain the annulus with its triangulation \(T\). \par	
	We let \(\tilde{T}\) be the triangulation of \(\mathcal{S}\) which coincides with the triangulation on \(R_i\) for each \(i \in \mathbb{Z}\). We see that \(\tilde{T}\) is stable under the shift \(\Sigma\) of arcs. It is not difficult to see that the quiver of this triangulation is exactly \(A_Q\), and that the automorphism \(\Sigma\) of \((A_Q)_0\) corresponds to the shift \(\Sigma\) of arcs. \par	
	\begin{figure}[ht]
		\begin{center}
			\def\svgwidth{0.92\textwidth}
			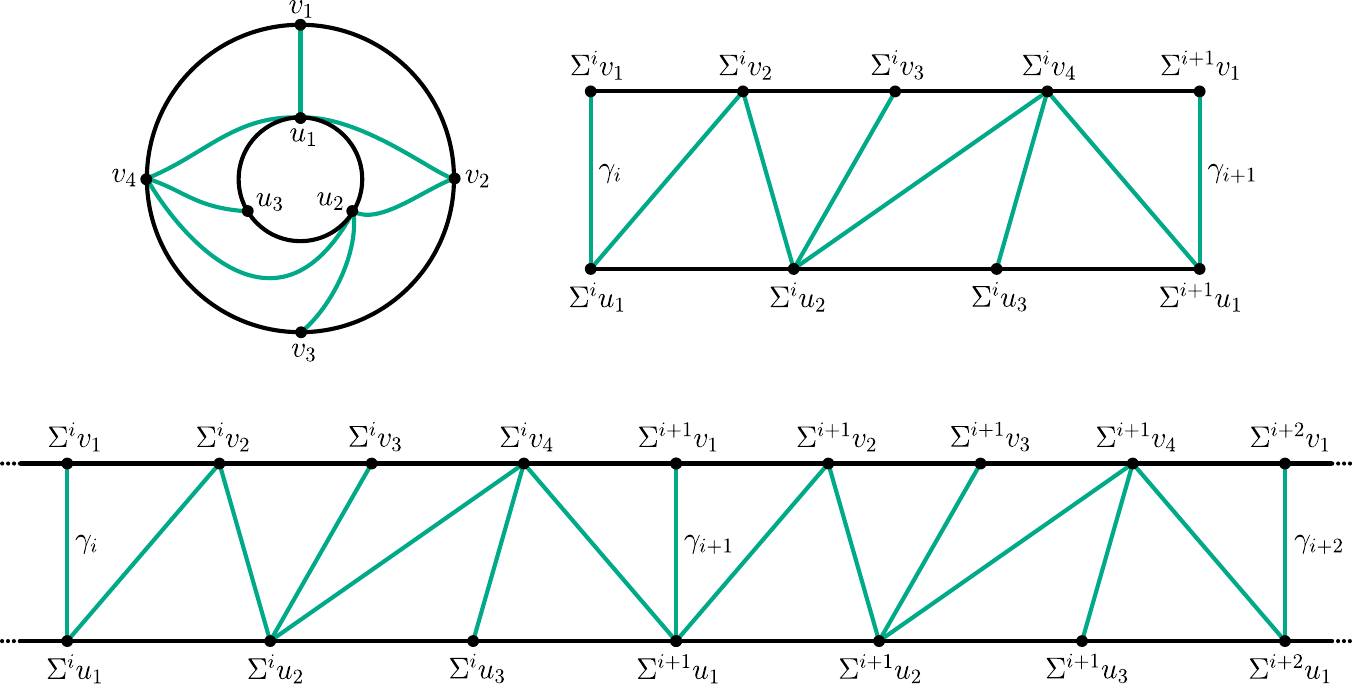
		\end{center}
		\caption{Constructing a triangulation of the infinite ribbon from a triangulation of the annulus}
	\end{figure}
	As \(A_Q\) is strongly \(G\)-admissible, we can mutate at \(\Sigma\)-orbits. This is equivalent to being able to flip simultaneously \(\Sigma\)-orbits of arcs in \(\tilde{T}\). It is clear that the new triangulation we obtain remains \(\Sigma\)-stable. In order to prove that \(A_Q\) is strongly globally foldable, following \Cref{remark-vitrual-2-cycles}, it is enough to prove that given a \(\Sigma\)-stable triangulation of \((\mathcal{M},\mathcal{S})\), its quiver will not contain virtual \(2\)-cycles. \par 
	A virtual \(2\)-cycle in a \(\Sigma\)-stable triangulation corresponds to a pair of arcs \(\beta_1\) and \(\beta_2\) belonging to a common \(\Sigma\)-orbit, sharing an end point, bounding a rectangle. Denote the end points of \(\beta_1\) by \(m_1\) and \(m_2\), with \(m_2\) being an end point of \(\beta_2\) as well. Without loss of generality \(\beta_2 = \Sigma^{a} \beta_1\) for \(a \geq 1\), and \(m_1 = v_i\). In particular, the end points of \(\beta_1\) are \(m_1 = v_i\) and \(m_2 = \Sigma^{a} m_1 = \Sigma^{a} v_i\), and the end points of \(\beta_2 = \Sigma^{a} \beta_1\) are \(\Sigma^{a} m_1 = \Sigma^{a} v_i\) and \(\Sigma^{a} m_2 = \Sigma^{2a} v_i\). \par
	Since \(\beta_1\) and \(\beta_2\) lie on a rectangle, there is a marked point \(m_3\) and arcs \(\delta_1\) and \(\delta_2\) belonging to a common \(\Sigma\)-stable triangulation, connecting \(m_1\) and \(\Sigma^{2a} m_1\) with \(m_3\), respectively. We consider the possible configurations of \(m_3\). 
	
	\begin{enumerate}
		\item[] Case 1: \(m_3 = \Sigma^{j} v_k\) and it lies to the left of \(m_1\). The triangulation is \(\Sigma\)-stable, \(\Sigma^{-a} \delta_2\) is also an arc in the triangulation. But \(\Sigma^{-a} \delta_2\) and \(\delta_2\) intersect. So this configuration is not possible.
		\item[] Case 2: \(m_3 = \Sigma^{j} v_k\) and it lies to the right of \(\Sigma^{a} m_2\). Similarly to the above case, \(\Sigma^a \delta_1\) and \(\delta_1\) intersect, so this is an impossible configuration. 
		\item[] Case 3: \(m_3 = \Sigma^{j} u_k\). In this case \(\Sigma^a \delta_1\) intersects \(\delta_2\), which is impossible. 
	\end{enumerate}
	See \Cref{three_cases} for a visual description of these cases. Overall, we have shown that it is impossible to have a \(\Sigma\)-stable triangulation whose quiver contains a virtual \(2\)-cycle. Therefore \((A_Q,\Sigma)\) is strongly globally foldable.  
\end{proof}

\begin{figure}[ht]
	\begin{center}
		\def\svgwidth{0.92\textwidth}
		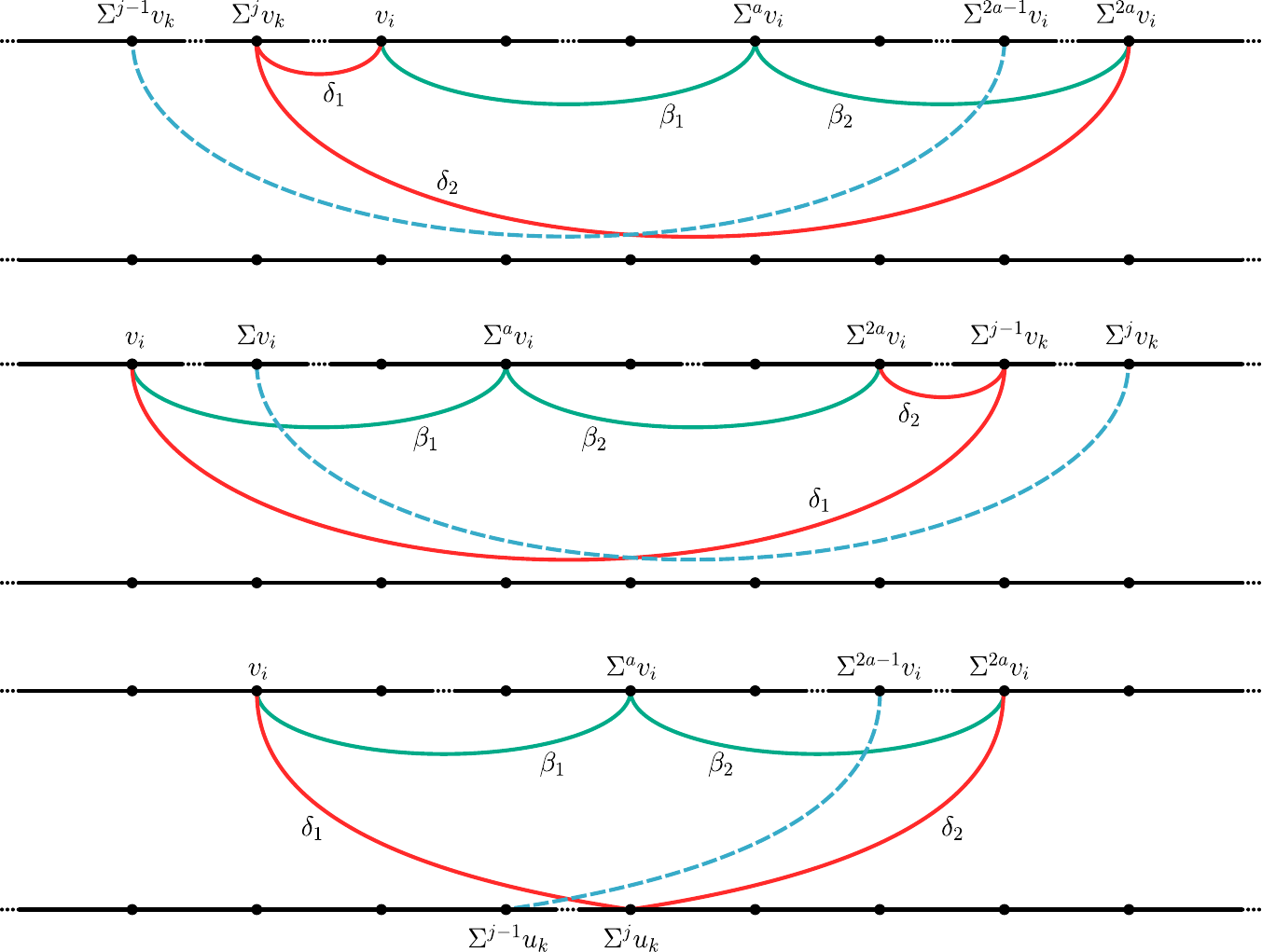
	\end{center}
	\caption{From top to bottom, cases 1-3 in the proof of \Cref{folding-theorem}}\label{three_cases}
\end{figure}

\begin{remark}
	From \Cref{tree-proposition} follows that a result analogous to \Cref{folding-theorem} for affine quivers different from  \(A\)-type cannot exist. That is to say, there is no infinite quiver \(\Gamma\) with a group \(G\) acting on its vertices such that \(\Gamma\) is strongly \(G\)-admissible and \(\Gamma^G\) is of affine-\(D\) or affine-\(E\) type.
\end{remark}

\begin{remark}
	The geometric model for \(A_\infty\) is the one considered in \cite{LiuPaq}, where they study the associated additive cluster category. See also \cite{GraGra}.
\end{remark}
\begin{remark}\label{stable-arc}
	It is not hard to see that any two \(\Sigma\)-stable triangulations of \((\mathcal{S},\mathcal{M})\) are connected by a finite sequence of \(\Sigma\)-orbit flips. Also, if \(\gamma\) is an arc which does not intersect \(\Sigma^{i}\gamma\) for any \(i\in\mathbb{Z}\), then it can be completed to a \(\Sigma\)-stable triangulation.
\end{remark}

\section{The category \(\mathcal{C}_1(\uqslinf)\)}\label{C1infty}

{In this section, we define the subcategory \(\mathcal{C}_1(\uqslinf)\) of the category \(\Oint(\uqslinf)\) and use the results of \cite{HL10} to identify its Grothendieck ring with a cluster algebra.

We remind the notation \(I_\infty = \mathbb{Z}\). We also set \((I_\infty)_0 = 2\mathbb{Z}\) and \((I_\infty)_1 = 2\mathbb{Z}+1\). \par
We define a quiver \(\Gamma_\infty\) as follows. Its vertex set is \(I_\infty \sqcup I_\infty'\) with \(I_\infty' = \mathbb{Z}\), so \((\Gamma_\infty)_0\) consists of two copies of \(\mathbb{Z}\). We denote by \(i\) the vertices in \(I_\infty\) and by \(i'\) the vertices in \(I_\infty'\). The arrows of \(\Gamma_\infty\) are given by
\begin{align*}
	&\Arr_{{\Gamma_\infty}}(i,j) = \begin{cases}
		\delta_{i,i+1}+\delta_{i,i-1} & i\in (I_\infty)_0 \\ 0 & i \in (I_\infty)_1
	\end{cases}, 
	&& \Arr_{{\Gamma_\infty}}(i',j') = 0, \\
	&\Arr_{{\Gamma_\infty}}(i',j) = \begin{cases}
		\delta_{i,j} & i\in (I_\infty)_0 \\ 0 & i \in (I_\infty)_1
	\end{cases},
	&& \Arr_{{\Gamma_\infty}}(i,j') = \begin{cases}
		0 & i\in (I_\infty)_0 \\ \delta_{i,j} & i \in (I_\infty)_1
	\end{cases}.
\end{align*}
We declare the vertices \(i'\) to be frozen. We think of \(\Gamma_\infty\) as a bi-infinite version of the quiver \(\Gamma_n\) introduced in \Cref{Recollection-C1}. \par
\begin{figure}[h]
	\begin{center}
		\[\begin{tikzcd}
			& \boxed{-3'} & \boxed{-2'} & \boxed{-1'} & \boxed{0'} & \boxed{1'} & \boxed{2'} & \boxed{3'} \\
			\cdots & {-3} & {-2} & {-1} & 0 & 1 & 2 & 3 & \cdots
			\arrow[from=1-3, to=2-3]
			\arrow[from=1-5, to=2-5]
			\arrow[from=1-7, to=2-7]
			\arrow[from=2-1, to=2-2]
			\arrow[from=2-2, to=1-2]
			\arrow[from=2-3, to=2-2]
			\arrow[from=2-3, to=2-4]
			\arrow[from=2-4, to=1-4]
			\arrow[from=2-5, to=2-4]
			\arrow[from=2-5, to=2-6]
			\arrow[from=2-6, to=1-6]
			\arrow[from=2-7, to=2-6]
			\arrow[from=2-7, to=2-8]
			\arrow[from=2-8, to=1-8]
			\arrow[from=2-9, to=2-8]
		\end{tikzcd}\]
	\end{center}
	\caption{The quiver \(\Gamma_\infty\). The square vertices are frozen}
\end{figure}
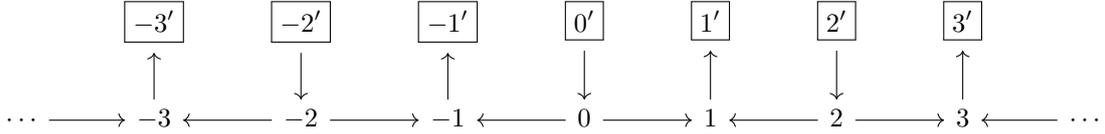 
We let \(\xi : I_\infty \rightarrow \{0,1\}\) be the map defined by \(\xi(i) = \begin{cases}
	0 & i \in (I_\infty)_0 \\
	1 & i \in (I_\infty)_1
\end{cases}\). \par 
We let \(\mathcal{M}_\infty \subseteq \Asub{\infty}\) be the subgroup of monomials generated by \(Y_{i,q^{\xi(i)}}\) and \(Y_{i,q^{\xi(i)+2}}\) for all \(i\in I_\infty\). We define the category \(\mathcal{C}_1(\uqslinf)\) to be the full subcategory of \(\Oint(\uqslinf)\) consisting of modules with finite composition series, such that every composition factor is isomorphic to \(L(m)\) for some \(m \in \mathcal{M}_\infty\). This is an analogue of the category \(\mathcal{C}_1(\uqsl{n+1})\) of \Cref{Recollection-C1}. \par
Following \cite{H05,H07}, the category of integrable (not necessarily finite length) \(\uqslinf\)-modules admits a fusion product.
\begin{proposition}
	Let \(V\) and \(W\) be two modules in \(\Oint(\uqslinf)\), so they have finite composition series. Then their fusion product \(V\ast W\) has a finite composition series. In particular, the fusion product on \(\Oint(\uqslinf)\) equips the Grothendieck group \(\mathcal{R}(\Oint(\uqslinf))\) with a commutative ring structure.
\end{proposition}
\begin{proof}
	Without loss of generality, we can assume \(V\) and \(W\) are simple, we write \(V = L(m_{\gamma_1})\) and \(W = L(m_{\gamma_2})\) for \(\gamma_1,\gamma_2 \in P_l^+\). Then \(\chi_q(V\ast W) = \chi_q(V)\cdot \chi_q(W)\), and it will be enough to show that \(\chi_q(V\ast W)\) contains finitely many dominant monomials. If \(\lambda \in P^+\) is such that \((V\ast W)_\lambda \neq 0\), then \(\lambda \in (\omega(m_{\gamma_1})+\omega(m_{\gamma_2})- Q^+) \cap P^+\). But \(\lvert (\omega(m_{\gamma_1})+\omega(m_{\gamma_2})- Q^+) \cap P^+ \rvert < \infty\), so there are finitely many dominant monomials appearing in \(\chi_q(V\ast W)\).
\end{proof}

\begin{lemma}\label{grothendieck-polynomial}
	The category \(\mathcal{C}_1(\uqslinf)\) is closed under the fusion product, and its Grothendieck ring \(\mathcal{R}(\mathcal{C}_1(\uqslinf))\) is isomorphic to the polynomial ring in countable infinitely many variables
	\begin{align*}
		\mathbb{Z}\left[[L(Y_{i,q^{\xi(i)}})],[L(Y_{i,q^{\xi(i)+2}})] \mid i \in I_\infty \right].
	\end{align*}
\end{lemma}
\begin{proof}
	This is completely analogous to \Cref{HL-grothendieck-polynomial}.
\end{proof}

We consider the algebra \(\mathcal{F}_\infty = \mathbb{Z}[f_i, x_i^{\pm 1} \mid i \in \mathbb{Z}]\). We let \(\mathcal{A}_\infty \subseteq \mathcal{F}_\infty\) be the cluster algebra associated to \(\Gamma_\infty\), with initial cluster \((f_i,x_i)_{i\in I_\infty}\) such that \(x_i\) corresponds to the mutable vertex \(i\) and \(f_i\) to the frozen vertex \(i'\). \par
It is clear that every sequence of mutation of \(\Gamma_\infty\) is supported on the vertices of \([-n,n]\) for \(n\) large enough. In particular, we deduce that the cluster variables are in bijection with the almost positive roots of the root system of \(A_\infty\), and these can be read from their denominator vectors. \par
For an almost positive root \(\alpha \in \Delta\), we let \(x[\alpha] \in \mathcal{A}_\infty\) be the associated (non-frozen) cluster variable. In particular \(x[-\alpha_i] = x_i\). \par
\begin{lemma}\label{cluster-polynomial}
	The cluster algebra \(\mathcal{A}_\infty\) is isomorphic to the polynomial algebra in countably infinitely many variables
	\begin{align*}
		\mathbb{Z}\left[x[-\alpha_i], x[\alpha_i] \mid i \in I_\infty\right].
	\end{align*}
\end{lemma}
\begin{proof}
	This follows from the discussion above and \Cref{HL-cluster-polynomial}
\end{proof}
To an almost positive root \(\alpha\in \Delta\) we associate a monomial \(m_\alpha \in \mathcal{M}_{\infty}\) as follows. We set
\begin{align*}
	m_{-\alpha_i} = \begin{cases}
		Y_{i,q^{\xi(i)+2}} & i \in (I_\infty)_0, \\
		Y_{i,q^{\xi(i)}} & i \in (I_\infty)_1,
	\end{cases} && \text{and} && m_{\alpha_i} = \begin{cases}
	Y_{i,q^{\xi(i)}} & i \in (I_\infty)_0, \\
	Y_{i,q^{\xi(i)+2}} & i \in (I_\infty)_1,
	\end{cases}
\end{align*}
for the simple roots. For a non-simple positive root \(\alpha = \sum_{i \in [i_1,i_2]} \alpha_i\) we set \(m_{\alpha} = \prod_{i\in[i_1,i_2]} m_{\alpha_i}\). \par
We let 
\begin{align*}
	\iota : \mathcal{A}_\infty \rightarrow \mathcal{R}(\mathcal{C}_1(\uqslinf))
\end{align*} 
be the ring homomorphism defined by \(\iota(x[\pm \alpha_i]) = [L(m_{\pm\alpha_i})]\). By \Cref{grothendieck-polynomial} and \Cref{cluster-polynomial}, the map \(\iota\) is a well defined algebra isomorphism. 
\begin{proposition}\label{slinfty-cluster}
	The isomorphism \(\iota\) satisfies
	\begin{align*}
		\iota(x[\alpha]) = [L(m_\alpha)] && \text{and} && \iota(f_i) = [L(Y_{i,q^{\xi(i)}}Y_{i,q^{\xi(i)+2}})]
	\end{align*}
	for any almost positive root \(\alpha\) and \(i \in \mathbb{Z}\). 
\end{proposition}
\begin{proof}
	This is a consequence of \Cref{hproposition-char}, \Cref{remark-char-product} and \Cref{HR-sln-cluster}
\end{proof}
\begin{remark}
	In the following, we will identify \(\mathcal{R}(\mathcal{C}_1(\uqslinf))\) with its image under the \(q\)-character map, and we may see \(\iota\) as a map from \(\mathcal{A}_\infty\) into \(\mathcal{Y}_\infty\).
\end{remark}
}
\section{Folding and \(\uqtor{2n}\)}\label{uqtor}
In this section, we follow the results of \cite{N11}, restricted to the \(A^{(1)}_n\)-case. Because \cite{N11} requires the Dynkin diagram to have no odd cycles, we actually only consider \(A^{(1)}_{2n-1}\), for \(n > 0\). In particular, we work with the algebras \(\uqtor{2n}\). We write \(I_{2n} = \mathbb{Z}/2n{\mathbb{Z}}\), \(\Asub{2n}\coloneq \tilde{A}(\uqtor{2n})\) and \(\mathcal{Y}_{2n}\coloneq \Ytor{2n}\).
\subsection{A subcategory of \(\Oint(\uqtor{2n})\)}
We would like to introduce a subcategory of \(\Oint(\uqtor{2n})\) similar to \(\mathcal{C}_1(\uqslinf)\) introduced before, but we need to be a bit more delicate with the definition. \par
We let \((I_{2n})_0\) be the subset of even integers and \((I_{2n})_1\) be the subset of odd integers of \(I_{2n}\).
We define the map \(\xi : I_{2n} \rightarrow \{0,1\}\) by
\begin{align*}
	\xi(i) = \begin{cases}
		0 & i \in (I_{2n})_0 \\
		1 & i \in (I_{2n})_1
	\end{cases}.
\end{align*}  
We let \(\mathcal{M} \subseteq \Asub{2n,\tor}\) be the subgroup generated by the monomials of the following form:
\begin{itemize}
	\item \(k_{\omega} Y_{i,q^{\xi(i)}}\) with \(\omega \in P\) such \(\omega(\alpha_j^\vee) = \delta_{i,j} \),
	\item \(k_{\omega} Y_{i,q^{\xi(i)+2}}\) with \(\omega \in P\) such \(\omega(\alpha_j^\vee) = \delta_{i,j} \),
\end{itemize}  
for \(i\in I_{2n}\). \par
We let \(\mathcal{C}_1(\uqtor{2n})\) be the full subcategory of \(\Oint(\uqtor{2n})\) of modules with finite composition series such that every composition factor \(S\) is isomorphic to \(L(m)\) for \(m \in \mathcal{M}\). \par
The proof of \Cref{HL-grothendieck-polynomial} shows that the category \(\mathcal{C}_1(\uqtor{2n})\) is closed under fusion product of \cite{H05,H07} and the tensor product of \cite{Lau25}. In particular, the Grothendieck group \(\mathcal{R}(\mathcal{C}_1(\uqtor{2n}))\) has a ring structure (the ring structure given by the fusion product and the tensor product is the same, see Section 6 of \cite{Lau25}).
\begin{remark}
	Contrary to the \(A\)-type case, given a monomial \(m \in \tilde{A}_{2n,\tor}\), the weight \(\omega(m)\) is not fully determined by \({(u_{i,a}(m))}_{i\in I_{2n}, a\in\mathbb{C}^*}\). In particular, if two monomials \(m,m' \in \tilde{A}_{2n,\tor}\) satisfy that \({(u_{i,a}(m))}_{i\in I_{2n}, a\in\mathbb{C}^*} = {(u_{i,a}(m'))}_{i\in I_{2n}, a\in\mathbb{C}^*}\), then \(m^{-1}m' = k_{t\delta}\) for some \(t\in \mathbb{C}\), where \(\delta = \sum_{i=0}^{2n-1} \alpha_i\). \par
	We will now show how to restrict our attention to a smaller class of monomials in a way that captures the structure of \(\mathcal{C}_1(\uqtor{2n})\). This will allow us to better interpret the results of \cite{H-folding} and \cite{N11}.
\end{remark}
For \(i\in I_{2n}\), let \(\omega_{i,\xi(i)}, \omega_{i,\xi(i)+2} \in P\) be such that \(k_{\omega_{i,\xi(i)}}Y_{\omega_{i,\xi(i)}}, k_{\omega_{i,\xi(i)+2}}Y_{\omega_{i,\xi(i)+2}} \in \mathcal{M}\). Let \(\mathcal{M}'\) be the subgroup of \(\mathcal{M}\) generated by \(k_{\omega_{i,\xi(i)}}Y_{\omega_{i,\xi(i)}}, k_{\omega_{i,\xi(i)+2}}Y_{\omega_{i,\xi(i)+2}}\). It follows that if \(m \in \mathcal{M}'\), then \(\omega(m)\) is determined by \({(u_{i,a}(m))}_{i\in I_n, a\in\mathbb{C}^*}\). \par
Let \(\mathcal{C}_1(\mathcal{M}')\) be the subcategory of \(\mathcal{C}_1(\uqtor{2n})\) consisting of modules with composition factors isomorphic to \(L(m)\) for some \(m\in \mathcal{M}'\). There is, a priori, no reason for \(\mathcal{C}_1(\mathcal{M}')\) to be closed under the fusion product, but we can show that \(\mathcal{M}'\) can be chosen such that it does happen.
\begin{lemma}
	For \(i\in I_{2n}\), there is a (non-unique) choice of \(\omega_{i,\xi(i)}, \omega_{i,\xi(i)+2} \in P\) such that \(\mathcal{C}_1(\mathcal{M}')\) is closed under the fusion product. 
\end{lemma}
\begin{proof}
	Following the argument of the proof of \Cref{HL-grothendieck-polynomial}, it is enough to show that if \(m \in \mathcal{M}'\) and \(A \in \mathbb{Z}[A_{i,a} \mid i\in I_{2n}, a\in\mathbb{C}^*]\) are such that \(mA^{-1}\) is a dominant monomial, then \(mA^{-1} \in \mathcal{M}'\).\par 
	So suppose \(m\) and \(A\) are as above, so that \(mA^{-1}\) is dominant. It follows from Section 7 of \cite{HL10} that \(A\) is actually a product of \(A_{i,q^{\xi(i)+1}}\) for \(i \in I_{2n}\). But as
	\begin{align*}
		A_{i,q^{\xi(i)+1}} = k_{\alpha_i}Y_{i,q^{\xi(i)}}Y_{i,q^{\xi(i)+2}}Y_{i-1,q^{\xi(i)+1}}^{-1}Y_{i+1,q^{\xi(i)+1}}^{-1},
	\end{align*}
	we see that if 
	\begin{align}\label{k-equation}
		k_{\alpha_i} = k_{\omega_{i,\xi(i)}} k_{\omega_{i,\xi(i)+2}}  k_{\omega_{i-1,\xi(i)+1}}^{-1}k_{\omega_{i+1,\xi(i)+1}}^{-1}
	\end{align} 
	holds, then \(mA^{-1} \in \mathcal{M}'\). We notice that for \(i\in (I_{2n})_0\), the term \(k_{\omega_{i,\xi(i)}}\) appears only in \(k_{\alpha_i}\). Similarly, for \(i\in (I_{2n})_1\), the term \(k_{\omega_{i,\xi(i)+2}}\) appears only in \(k_{\alpha_i}\). We can, therefore, choose \(\omega_{i,\xi(i)}\) and \(\omega_{i,\xi(i)+2}\) such that \eqref{k-equation} holds, and the claim of the lemma follows.
\end{proof}
From now on let us fix \(\mathcal{M}'\) such that \(C_1(\mathcal{M}')\) is closed under the fusion product. We obtain a ring embedding 
\begin{align*}
	\mathcal{R}(\mathcal{C}_1(\mathcal{M}')) \hookrightarrow \mathcal{R}(\mathcal{C}_1(\uqtor{2n})).
\end{align*} 
It should be noted that every simple module \(L(m)\) in \(\mathcal{C}_1(\uqtor{2n})\) is isomorphic to a product \(L(k_{t\delta})\ast L(m')\) for unique \(m' \in \mathcal{M}'\) and \(t \in \mathbb{C}\). So the structure of \(\mathcal{C}_1(\uqtor{2n})\) is determined by the structure of \(\mathcal{C}_1(\mathcal{M}')\), and we have an isomorphism of rings
\begin{align*}
	\theta : \mathbb{Z}[\mathbb{C}] \otimes_\mathbb{Z} \mathcal{R}(\mathcal{C}_1(\mathcal{M}')) \rightarrow \mathcal{R}(\mathcal{C}_1(\uqtor{2n}))
\end{align*} 
given by
\begin{align*}
	\theta(t\otimes [L(m')]) = [L(k_{t\delta})\ast L(m')] = [L(k_{t\delta}\cdot m')],
\end{align*}
where \(\mathbb{Z}[\mathbb{C}]\) is the group algebra of the additive group of complex numbers. The multiplication in \(\mathbb{Z}[\mathbb{C}] \otimes_\mathbb{Z} \mathcal{R}(\mathcal{C}_1(\mathcal{M}'))\) is given by 
\begin{align*}
	(t\otimes [L(m')]) \cdot (t'\otimes [L(m'')]) = (t+t')\otimes [L(m')\ast L(m'')].
\end{align*}
This justifies working with the category \(\mathcal{C}_1(\mathcal{M}')\) and its Grothendieck ring. \par
Because of how we defined \(\mathcal{M}'\), for simplicity of notation, we can drop the \(k_\omega\) variables when describing simple modules and \(q\)-characters of modules in \(\mathcal{C}_1(\mathcal{M}')\). We thus ignore from now on the variables \(k_\omega\) appearing in \(\Asub{2n}\) and \(\Ysub{2n}\). We set \(\mathcal{K}_{2n,\tor} \coloneq \mathcal{R}(\mathcal{C}_1(\mathcal{M}'))\) and identify it with its image in \(\Ysub{2n}\).

\subsection{The map \(\phi_{2n}\) and folding of \(q\)-characters}

\begin{definition}[\cite{H-folding}]
	Let \(\phi_{2n} : \Asub{\infty} \rightarrow \Asub{2n}\) be the group homomorphism given by
	\begin{align*}
		\phi_{2n}(Y_{i,a}) = Y_{[i],a},
	\end{align*}
	where \(i \in \mathbb{Z}\), \([i] \in \mathbb{Z}/2n\mathbb{Z}\) is its equivalence class modulo \(2n\), and \(a \in \mathbb{C}^*\).
\end{definition}

\begin{remark}
	Because we ignore the variables \(k_\omega\), there is ambiguity in how \(\phi_{2n}\) should extend to a morphism \(\Ysub{\infty} \rightarrow \mathcal{Y}_{2n}\). However, we will require that \(\phi_{2n}(A_{i,a})=A_{[i],a}\), so that in practice we are able to apply the map \(\phi_{2n}\) to \(q\)-characters of simple \(\uqslinf\)-modules, in the following sense. Let \(V\) be a simple integrable \(\uqslinf\)-module, with highest dominant monomial \(m\). Therefore its \(q\)-character is of the form
	\begin{align*}
		\chi_q(V) = \sum_{m' \leq m} a_{m'} m'
	\end{align*} 
	with \(m' \in \Asub{\infty}\), \(a_{m'} \in \mathbb{N}\) and \(a_m = 1\). We shall understand \(\phi_{2n}(\chi_q(V))\) as
	\begin{align*}
		\phi_{2n}(\chi_q(V)) = \sum_{m' \leq m} a_n k_{\omega_{m'}} \phi_{2n}(m') \in \Ysub{2n},
	\end{align*}
	where \(\omega_{m}\in P\) is such that \(k_{\omega_{m}}\phi_{2n}(m) \in \Asub{2n}\) (so \(\omega_m\) is determined by \(m\) up to \(t\delta\)), and \(\omega_{m'} \in P\) is the unique element satisfying \(k_{\omega_{m'}} \phi_{2n}(m') \leq k_{\omega_{m}}\phi_{2n}(m)\) in \(\Asub{2n}\). If for every dominant monomial \(m\) we fix such an \({\omega_{m}}\), we can extend \(\phi_{2n}\) linearly to obtain a homomorphism of abelian groups from \(\mathcal{R}(\Oint(\uqslinf))\) to \(\Ysub{2n}\). \par
	In \cite{H-folding}, Hernandez shows that \(\phi_{2n}\) takes \(q\)-characters of integrable \(\uqslinf\)-modules to virtual \(q\)-characters in \(\mathcal{R}(\Oint(\uqtor{2n}))\). More explicitly, this means that if \(V\) is a module in \(\Oint(\uqslinf)\), then
	\begin{align*}
		\phi_{2n}(\chi_q(V)) = \sum_{\substack{m'\in \Asub{2n},\\ m' \text{ dominant}}} a_{m'} \chi_q(L(m')),
	\end{align*}
	where \(a_{m'} \in \mathbb{Z}\), and \(a_{m'} = 0\) for almost all \(m'\). \par 
	In particular, by fixing appropriate values of \({\omega_m}\), we can understand \(\phi_{2n}\) as a ring homomorphism from \(\mathcal{R}(\mathcal{C}_1(\uqslinf))\) to \(\mathcal{R}(\mathcal{C}_1(\mathcal{M}'))\).
\end{remark}

In \cite{H-folding}, Hernandez conjectures the following.

\begin{conjecture}[Conjecture 5.3 in \cite{H-folding}]\label{conj-fold}
	The map \(\phi_{2n}\) takes \(q\)-characters of integrable \(\uqslinf\)-modules to \(q\)-characters of actual \(\uqtor{2n}\)-modules. 
\end{conjecture}

He then shows that the conjecture holds for a specific class of integrable \(\uqslinf\)-modules.

\begin{theorem}[Theorem 4.2 in \cite{H-folding}]\label{Hphi-theorem}
	Let \(i \in \mathbb{Z}\), let \(a \in \mathbb{C}^*\) and \(k \geq 0\) and set 
	\begin{align*}
		m_{i,a,k}^\mathrm{KR} = \prod_{j=1}^{k} Y_{i,aq^{2(j-1)}} \in \Asub{\infty} \qquad \text{and} \qquad 		m_{[i],a,k}^\mathrm{KR} = \prod_{j=1}^{k} Y_{[i],aq^{2(j-1)}} \in \Asub{2n}.
	\end{align*}
	Then \(\phi_{2n}(\chi_q(L(m_{i,a,k}^\mathrm{KR}))) = \chi_q(L(m_{[i],a,k}^\mathrm{KR}))\). That is, the map \(\phi_{2n}\) takes \(q\)-characters of Kirillov--Reshetikhin \(\uqslinf\)-modules to \(q\)-characters of Kirillov--Reshetikhin \(\uqtor{2n}\)-modules. \qed
\end{theorem}

It is worth noting that outside the case of Kirillov--Reshetikhin modules, the conjecture remains open. Later in this paper, we will give examples of non-Kirillov--Reshetikhin modules that satisfy this conjecture (see \Cref{cor-new-cases} and \Cref{further}).

\subsection{Monoidal categorification of \(C_1(\mathcal{M}')\)}

We now return to our main goal for this section. For \(n > 0\), we let \(\tilde{\Gamma}_{2n}\) be the quiver with \(4n\) vertices, denoted by \(\{1,...,2n\}\cup\{1',...,2n'\}\). We think of them as elements of \(I_{2n}\coloneq \mathbb{Z}/2n\mathbb{Z}\). The arrows of \(\tilde{\Gamma}_{2n}\) are given by 

\begin{align*}
	&\Arr_{\tilde{\Gamma}_{2n}}(i,j) = \begin{cases}
		\delta_{i,j+1}+\delta_{i,j-1} & i\in (I_{2n})_0 \\ 0 & i \in (I_{2n})_1
	\end{cases}, 
	&& \Arr_{\tilde{\Gamma}_{2n}}(i',j') = 0, \\
	&\Arr_{\tilde{\Gamma}_{2n}}(i',j) = \begin{cases}
		\delta_{i,j} & i\in (I_{2n})_0 \\ 0 & i \in (I_{2n})_1
	\end{cases},
	&& \Arr_{\tilde{\Gamma}_{2n}}(i,j') = \begin{cases}
		0 & i\in (I_{2n})_0 \\ \delta_{i,j} & i \in (I_{2n})_1
	\end{cases}.
\end{align*}
We set the vertices \(i'\) to be frozen. 
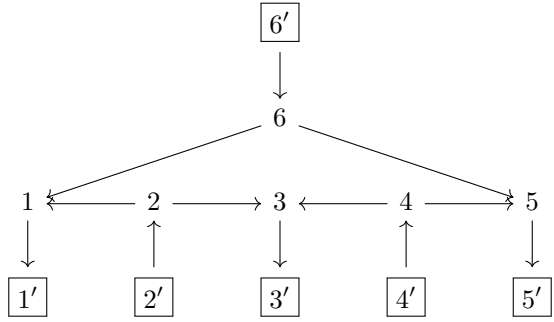
\begin{figure}[h]
	\begin{center}
		\[\begin{tikzcd}
			&& \boxed{6'} \\
			&& 6 \\
			1 & 2 & 3 & 4 & 5 \\
			\boxed{1'} & \boxed{2'} & \boxed{3'} & \boxed{4'} & \boxed{5'}
			\arrow[from=1-3, to=2-3]
			\arrow[from=2-3, to=3-1]
			\arrow[from=2-3, to=3-5]
			\arrow[from=3-1, to=4-1]
			\arrow[from=3-2, to=3-1]
			\arrow[from=3-2, to=3-3]
			\arrow[from=3-3, to=4-3]
			\arrow[from=3-4, to=3-3]
			\arrow[from=3-4, to=3-5]
			\arrow[from=3-5, to=4-5]
			\arrow[from=4-2, to=3-2]
			\arrow[from=4-4, to=3-4]
		\end{tikzcd}\]
	\end{center}
	\caption{The quiver \(\tilde{\Gamma}_{2n}\) with \(n = 3\). The square vertices are frozen}
\end{figure} 

We can identify \(\tilde{\Gamma}_{2n}\) with the folding \(\Gamma_{\infty}^G\) of \(\Gamma_{\infty}\) with respect to the group \(G\subseteq \Sym((\Gamma_{\infty})_0)\), generated by the shift \(\Sigma\) by \(2n\) of \((\Gamma_{\infty})_0\), which is given by
\begin{align*}
	\Sigma(i) = i+2n \qquad \text{and} \qquad \Sigma(i') = (i+2n)'
\end{align*}
for \(i \in I_\infty\). \par 
We consider the algebra \(\mathcal{F}_{2n,\tor} = \mathbb{Z}[f_i,x_i^{\pm1} \mid i \in I_{2n}]\). We let \(\mathcal{A}_{2n,\tor} \subseteq \mathcal{F}_{2n,\tor}\) be the cluster algebra associated to \(\tilde{\Gamma}_{2n}\), with initial cluster \((x_i,f_i)_{i\in I_{2n}}\) such that \(x_i\) corresponds to the vertex \(i\) and \(f_i\) corresponds to the vertex \(i'\).
\begin{theorem}[Section 5, Propositions 7.5, 7.6 in \cite{N11}]\label{tor-cluster}
	There is an isomorphism of algebras
	\begin{align*}
		\iota : \mathcal{A}_{2n,\tor} \rightarrow \mathcal{K}_{2n,\tor}
	\end{align*}
	such that 
	\begin{align*}
		\iota(x_i) = \begin{cases}
			\chi_q(L(Y_{i,q^{\xi(i)+2}})) & \text{for } i \in (I_{2n})_0 \\
			\chi_q(L(Y_{i,q^{\xi(i)}})) & \text{for } i \in (I_{2n})_1
		\end{cases} \qquad \text{and} \qquad \iota(f_i) = \chi_q(L(Y_{i,q^{\xi(i)}}Y_{i,q^{\xi(i)+2}})).
	\end{align*}
	Every non-initial cluster variable is of the form
	\begin{align*}
		\chi_q\left(L\left(\prod_{i \in (I_{2n})_0}Y_{i,q^{\xi(i)}}^{a_i}\prod_{i \in (I_{2n})_1}Y_{i,q^{\xi(i)+2}}^{b_i}\right)\right)
	\end{align*} 
	for some \(a_i,b_i \in \mathbb{N}\).\par
	Finally, let \(m_1,...,m_s\) be dominant monomials such that \(\chi_q(L(m_1)),...,\chi_q(L(m_s))\) correspond to cluster variables. Then 
	\begin{align*}
		\prod_{i=1}^s\chi_q(L(m_i)) = \chi_q\left(L\left(\prod_{i=1}^sm_i\right)\right)
	\end{align*}
	if and only if \(\chi_q(L(m_1)),...,\chi_q(L(m_s))\) belong to the same cluster. \qed
\end{theorem}
\begin{remark}
	We can rephrase the last part of the theorem in the following way: the simple modules \(L(m_1),...,L(m_s)\) correspond to cluster variables which belong to a common cluster if and only if 
	\begin{align*}
		L(m_1)\ast \cdots \ast L(m_s)
	\end{align*} 
	is a simple module, if and only if \(\prod_{i=1}^{s}[L(m_i)]\) is a class of a simple module. 
\end{remark}

In what follows, using the isomorphisms \(\iota\), we will identify \(q\)-characters with their the corresponding classes in the Grothendieck ring, and with the elements in the associated cluster algebra. 
\subsection{Folding \(\mathcal{A}_\infty\) onto \(\mathcal{A}_{2n,\tor}\)}
As we remarked earlier, the quiver \(\tilde{\Gamma}_{2n}\) is the folding of \(\Gamma_{\infty}\) with respect to the group \(G\) generated by the shift \(\Sigma\). By \Cref{folding-theorem}, the mutable part of \(\Gamma_\infty\) is strongly globally foldable with respect to \(G\), and folds onto the mutable part of \(\tilde{\Gamma}_{2n}\). Our goal is to use the results above to show that the whole ice quiver \(\Gamma_{\infty}\) is strongly globally foldable with respect to \(G\). \par
Before doing so, we need the following technical lemma. For \(L \in \Ysub{2n}\), we write \(D(L) \subseteq \Asub{2n}\times \mathbb{Z}_{\neq 0}\) for the set of monomials appearing in \(L\) with their multiplicity. This means that \((m',a_{m'}) \in D(L)\) if and only if the monomial \(m'\) appears in \(L\) with coefficient \(a_{m'} \neq 0\). We will say that \(L\) is \(\mu\)-dominated with \(\mu\)-dominant weight \(m \in \mathcal{M}'\) if
\begin{enumerate}
	\item \(m\) is a dominant monomial such that \((m,1) \in D(L)\),
	\item if \((m',a_{m'}) \in D(L)\) then \(m' \leq m\),
	\item \(m\cdot Y_{i,\xi(i)}^{-1}Y_{i,\xi(i)+2}^{-1}\) is not dominant for any \(i \in I_{2n}\).
\end{enumerate}
For example, following \Cref{tor-cluster}, every simple module corresponding to a cluster variable in \(\mathcal{K}_{2n,\tor}\) is \(\mu\)-dominated, with \(\mu\)-dominant weight one of the following:
\begin{itemize}
	\item \(Y_{i,q^{\xi(i)+2}}\) with \(i \in (I_{2n})_0\),
	\item \(Y_{i,q^{\xi(i)}}\) with \(i \in (I_{2n})_1\),
	\item \(\prod_{i \in (I_{2n})_0}Y_{i,q^{\xi(i)}}^{a_i}\prod_{i \in (I_{2n})_1}Y_{i,q^{\xi(i)+2}}^{b_i}\) with \(a_i,b_i \in \mathbb{N}\).
\end{itemize}
\begin{lemma}\label{lemma-frozen-determined}
	Let \(M,S_1,S_2 \in \mathcal{K}_{2n,\tor}\) be \(q\)-characters of simple integrable \(\uqtor{2n}\)-modules, and let \(F_1,F_2 \in \mathcal{K}_{2n,\tor}\) be a pair of cluster monomials in frozen variables such that
	\begin{align}\label{first_lemma_equation}
		N\cdot M = S_1\cdot F_1 + S_2 \cdot F_2
	\end{align}
	for some \(\mu\)-dominated \(N \in \Ysub{2n}\). Suppose \(F_1',F_2' \in \mathcal{K}_{2n,\tor}\) is another pair of cluster monomials in frozen variables such that
	\begin{align*}
		N'\cdot M = S_1\cdot F_1' + S_2 \cdot F_2'
	\end{align*}
	for some \(\mu\)-dominated \(N' \in \Ysub{2n}\). Then \(F_1 = F_1'\) or \(F_2 = F_2'\).
\end{lemma}
\begin{proof}
	Write \(m\), \(s_1\) and \(s_2\) for the highest dominant monomials of \(M\), \(S_1\) and \(S_2\), respectively. Denote by \(k\) the \(\mu\)-dominant monomial of \(N\). Finally, write \(l_1\) and \(l_2\) for the highest dominant monomials of \(F_1\) and \(F_2\), respectively. \par
	We define a group morphism \(d:\mathcal{M}'\rightarrow \mathbb{Z}\) by
	\begin{align*}
	d(Y_{i,q^{\xi(i)}}) = \begin{cases}
		-1 & i \in (I_{2n})_0\\
		1 & i \in (I_{2n})_1
	\end{cases} \qquad \text{and} \qquad d(Y_{i,q^{\xi(i)+2}}) = \begin{cases}
		1 & i \in (I_{2n})_0\\
		-1 & i \in (I_{2n})_1
	\end{cases}.
	\end{align*}
	We can compute easily 
	\begin{align*}
		d(Y_{i,q^{\xi(i)}}Y_{i,q^{\xi(i)+2}}) = d(Y_{i,q^{\xi(i)}})+d(Y_{i,q^{\xi(i)+2}}) = 0
	\end{align*}
	for every \(i \in I_{2n}\), so that \(d(l_1)=d(l_2) = 0\). \par
	Similarly, we can compute
	\begin{align*}
		d(A_{i,\xi(i)+1}^{-1}) & = d(Y_{i,q^{\xi(i)}}^{-1}Y_{i,q^{\xi(i)+2}}^{-1}Y_{i-1,q^{\xi(i)+1}}Y_{i+1,q^{\xi(i)+1}}) \\
		& = -d(Y_{i,q^{\xi(i)}}Y_{i,q^{\xi(i)+2}})+d(Y_{i-1,q^{\xi(i)+1}})+d(Y_{i+1,q^{\xi(i)+1}}) \\
		& = d(Y_{i-1,q^{\xi(i)+1}})+d(Y_{i+1,q^{\xi(i)+1}}) \\
		& = \begin{cases}
			d(Y_{i-1,q^{\xi(i-1)+2}})+d(Y_{i+1,q^{\xi(i+1)+2}}) & i \in (I_{2n})_0\\
			d(Y_{i-1,q^{\xi(i-1)}})+d(Y_{i+1,q^{\xi(i+1)}}) & i \in (I_{2n})_1
		\end{cases}.
	\end{align*}
	Therefore \(d(A_{i,\xi(i)+1}^{-1}) = -2 \) for any \(i \in I_{2n}\). \par
	We deduce that if \(m'\) is any dominant monomial of \(N\cdot M\), then \(d(m') \leq d(k\cdot m)\). Moreover, \(d(m') = d(k \cdot m)\) if and only \(m' = k \cdot m\). \par
	By \eqref{first_lemma_equation}, we know \(s_1\cdot l_1\) and \(s_2\cdot l_2\) are monomials appearing in \(N\cdot M\), and exactly one of them equals \(k\cdot m\). We deduce \(d(s_1) = d(s_1\cdot l_1) < d(s_2\cdot l_2) = d(s_2)\) or \(d(s_2) = d(s_2\cdot l_2) < d(s_1\cdot l_1) = d(s_1)\). Without loss of generality, assume the latter holds. Then \(k\cdot m = s_1 \cdot l_1\), which is to say \(k = m^{-1}\cdot s_1\cdot l_1\). As
	\begin{align*}
		l_1 =  \prod_{i \in I_{2n}}(Y_{i,q^{\xi(i)}}Y_{i,q^{\xi(i)+2}})^{a_i}
	\end{align*}
	for some \(a_i \in \mathbb{N}\), and because \(k\) is a \(\mu\)-dominant monomial, the values \(a_i\) are completely determined by \(m^{-1}\cdot s_1\). This means \(l_1\) is completely determined by \(m^{-1}\cdot s_1\). \par 
	Now let \(l_1'\) and \(l_2'\) be the highest dominant monomials of \(F_1'\) and \(F_2'\), respectively. By a completely analogous argument for the equation 
	\begin{align*}
		N'\cdot M = S_1\cdot F_1' + S_2 \cdot F_2',
	\end{align*}
	we deduce that \(l_1 = l_1'\). This shows that \(F_1 = F_1'\), which concludes the proof.
\end{proof}

We now prove the main theorem of this section, which is also the principal result of this paper. The main takeaway of the theorem is that simple modules of \(\uqslinf\) are related to simple modules of \(\uqtor{2n}\), via the map \(\phi_{2n}\). Moreover, this map is a folding map of the associated cluster algebras. \par
In the proof, given a cluster \(X\) of a cluster algebra, we will denote by \(X_v\) the cluster variable in \(X\) corresponding to the vertex \(v\).

\begin{theorem}\label{theorem-ice-folding}
	The ice quiver \(\Gamma_\infty\) is strongly globally foldable with respect to \(G\). \par
	Let \(\psi : \mathcal{F}_\infty \rightarrow \mathcal{F}_{2n,\tor}\) be the homomorphism defined by \(\psi(f_i) = f_{[i]}\) and \(\psi(x_i) = x_{[i]}\) for \(i\in\mathbb{Z}\) (see \Cref{proposition-cluster-folding}). If \(\chi_q(L(m))\) is a cluster variable belonging to an orbit-cluster of \(\mathcal{A}_\infty\), then the folded cluster variable satisfies
	\begin{align*}
		\psi(\chi_q(L(m))) = \phi_{2n}(\chi_q(L(m))).
	\end{align*}
\end{theorem}

\begin{proof}
	Following \Cref{strongly-gf}, in order to show that \(\Gamma_\infty\) is strongly globally foldable, we need to prove that given a finite sequence \(K_1,...,K_{l}\) of mutable \(G\)-orbits in \(\Gamma_\infty\), such that \((\hmu_{K_{i}}\circ\cdots\circ\hmu_{K_1}) (\Gamma_{\infty})\) is strongly \(G\)-admissible for \(1 \leq i \leq l-1\), then the quiver \((\hmu_{K_l}\circ\cdots\circ\hmu_{K_1}) (\Gamma_{\infty})\) is also strongly \(G\)-admissible. \par
	In the situation above, an immediate consequence of the proof of \Cref{proposition-cluster-folding} is that if \(X\) is the orbit-cluster in \(\mathcal{A}_\infty\) corresponding to the sequence of orbit-mutation at \(K_1,...,K_{l}\) and \(\tilde{X}\) is the cluster in \(\mathcal{A}_{2n,\tor}\) corresponding the the sequence of mutation at \(K_{1},...,K_{l}\), then \(\tilde{X}_{[i]}= \psi(X_i)=\phi_{2n}(X_i)\). We will use this fact in the course of the proof. Moreover, it shows that in order to prove the theorem, it will suffice to show that \(\Gamma_\infty\) is strongly globally foldable. \par
	So let us show that \(\Gamma_{\infty}\) is strongly globally foldable. \par
	Let \(K_1,...,K_{l}\) be a finite sequence of mutable \(G\)-orbits in \(\Gamma_\infty\), such that \((\hmu_{K_{i}}\circ\cdots\circ\hmu_{K_1}) (\Gamma_{\infty})\) is strongly \(G\)-admissible for \(1 \leq i \leq l-1\). Let \(X\) and \(\tilde{X}\) be the cluster of \(\mathcal{A}_\infty\) and \(\mathcal{A}_{2n,\tor}\), respectively, corresponding to sequence of mutations at \(K_1,...,K_l\). \par 
	To show that \((\hmu_{K_l}\circ \cdots \circ\hmu_{K_1})(\Gamma_{\infty})\) is strongly \(G\)-admissible, following \Cref{mutation-behavior} and \Cref{remark-vitrual-2-cycles}, it suffices to verify 
	\begin{align}\label{suffices-to-verify}
		\Arr_{(\hmu_{K_l}\circ \cdots \circ \hmu_{K_1})(\Gamma_\infty)}(Gv_1,v_2) = \Arr_{(\mu_{K_l}\circ \cdots \circ \mu_{K_1})(\Gamma_{2n})}(Gv_1,Gv_2)
	\end{align}
	for any \(v_1,v_2 \in (\Gamma_\infty)_0\). According to \Cref{folding-theorem}, this equality holds for \(v_1,v_2\) mutable. So we only need to show the equality for either \(v_1\) frozen or \(v_2\) frozen. \par
	We denote by \(f_j\) the frozen variable of \(\mathcal{A}_{\infty}\) corresponding to vertex \(j'\) of \(\gamma_{\infty}\). Similarly we denote by \(f_{[j]}\) the frozen variables of \(\mathcal{A}_{2n,\tor}\) corresponding to vertex \([j]'\) of \(\tilde{\Gamma}_{2n}\).
	Fix some \(i \in \mathbb{Z}\), and mutate \((\hmu_{K_l}\circ \cdots \circ\hmu_{K_1})(\Gamma_{\infty})\) at \(i\) and \((\mu_{K_l}\circ \cdots \circ \mu_{K_1})(\Gamma_{2n})\) at \([i]\). The corresponding exchange relations are
	\begin{align}\label{exchange_infty}
		X'_i\cdot X_i = \prod_{j\in \mathbb{Z}}X_j^{a_j} f_j^{b_j} + \prod_{j\in \mathbb{Z}}X_j^{c_j} f_j^{d_j}, 
	\end{align}
	\begin{align}\label{exchange_tor}
		\tilde{X}'_{[i]} \cdot \tilde{X}_{[i]} = \prod_{[j]\in \mathbb{Z}/2n\mathbb{Z}}\tilde{X}_{[j]}^{a_{[j]}} f_{[j]}^{b_{[j]}} + \prod_{[j]\in \mathbb{Z}/2n\mathbb{Z}}\tilde{X}_{[j]}^{c_{[j]}} f_{[j]}^{d_{[j]}},
	\end{align}
	with 
	\begin{align*}
		a_j = \Arr_{(\hmu_{K_l}\circ \cdots \circ \hmu_{K_1})(\Gamma_\infty)}(i,j), && b_j = \Arr_{(\hmu_{K_l}\circ \cdots \circ \hmu_{K_1})(\Gamma_\infty)}(i,j'), \\
		c_j = \Arr_{(\hmu_{K_l}\circ \cdots \circ \hmu_{K_1})(\Gamma_\infty)}(j,i), && d_j = \Arr_{(\hmu_{K_l}\circ \cdots \circ \hmu_{K_1})(\Gamma_\infty)}(j',i), 
	\end{align*}
	and
	\begin{align*}
		a_{[j]} = \Arr_{(\mu_{K_l}\circ \cdots \circ \mu_{K_1})(\Gamma_{2n,\tor})}([i],[j]), && b_{[j]} = \Arr_{(\mu_{K_l}\circ \cdots \circ \mu_{K_1})(\Gamma_{2n,\tor})}([i],[j]'), \\
		c_{[j]} = \Arr_{(\mu_{K_l}\circ \cdots \circ \mu_{K_1})(\Gamma_{2n,\tor})}([j],[i]), && d_{[j]} = \Arr_{(\mu_{K_l}\circ \cdots \circ \mu_{K_1})(\Gamma_{2n,\tor})}({[j]}',[i]).
	\end{align*}
	From \Cref{mutation-behavior}, \Cref{remark-vitrual-2-cycles} and \Cref{folding-theorem}, we see that
	\begin{align*}
		\sum_{s\in j\mathbb{Z}} a_s = a_{[j]}, && \sum_{s\in j\mathbb{Z}} b_s \geq  b_{[j]}, && 
		\sum_{s\in j\mathbb{Z}} c_s = c_{[j]}, && \sum_{s\in j\mathbb{Z}} d_s \geq d_{[j]},
	\end{align*}
	for every \(j \in \mathbb{Z}\).
	Moreover, from \Cref{remark-vitrual-2-cycles} we deduce for every \(j \in \mathbb{Z}\) that
	\begin{align}\label{iff-arrows}
		\sum_{s\in j\mathbb{Z}} b_s =  b_{[j]} \iff \sum_{s\in j\mathbb{Z}} d_s = d_{[j]}.
	\end{align}	
	Applying the map \(\phi_{2n}\) to \eqref{exchange_infty}, we get
	\begin{align}\label{exchange_phi}
		\phi_{2n}(X_i')\cdot \tilde{X}_{[i]} = \prod_{[j]\in \mathbb{Z}/2n\mathbb{Z}}\tilde{X}_{[j]}^{a_{[j]}} f_{[j]}^{\sum_{s\in [j]}b_{s}} + \prod_{[j]\in \mathbb{Z}/2n\mathbb{Z}}\tilde{X}_{[j]}^{c_{[j]}} f_{[j]}^{\sum_{s\in [j]}d_{s}}.
	\end{align}
	Both \(\phi_{2n}(X_i')\) and \(\tilde{X}'_i\) are \(\mu\)-dominated in the sense of \Cref{lemma-frozen-determined}. Applying \Cref{lemma-frozen-determined} to the pair of equations \eqref{exchange_tor} and \eqref{exchange_phi}, we deduce \(\sum_{s\in j\mathbb{Z}} b_s =  b_{[j]}\) for every \(j \in \mathbb{Z}\) or \(\sum_{s\in j\mathbb{Z}} d_s = d_{[j]}\) for every \(j \in \mathbb{Z}\). \par
	But \eqref{iff-arrows} tells us both hold, hence \eqref{suffices-to-verify} is satisfied. Therefore \((\hmu_{K_l}\circ \cdots \circ \hmu_{K_1})(\Gamma_{\infty})\) is strongly \(G\)-admissible. This concludes the proof.
\end{proof}

\begin{corollary}\label{cor-ice-folding}
	Let \(\chi_q(L(m))\) be a cluster monomial belonging to an orbit-cluster of \(\mathcal{A}_\infty\). The folded cluster variable satisfies
	\begin{align*}
		\psi(\chi_q(L(m))) = \phi_{2n}(\chi_q(L(m))) = \chi_q(L(\phi_{2n}(m))).
	\end{align*}
\end{corollary}
\begin{proof}
	We only need to verify the last equality. \par
	From \Cref{tor-cluster} and \Cref{theorem-ice-folding}, we know that 
	\begin{align*}
		\phi_{2n}(\chi_q(L(m))) = \chi_q(L(m'))
	\end{align*} 
	for some \(m' \in \mathcal{M}'\). But \(\phi_{2n}(A_{i,a}) = A_{[i],a}\) for any \(i\in \mathbb{Z}\) and \(a \in \mathbb{C}^*\). In particular, \(\phi_{2n}(m)\) is a maximal monomial of \(L(m')\) with respect to the Nakajima partial order. The equality \(m' = \phi_{2n}(m)\) then follows.
\end{proof}

An immediate consequence is the following. 

\begin{corollary}\label{cor-new-cases}
	The simple \(\uqslinf\)-modules which correspond to cluster monomials of \(\mathcal{A}_\infty\) satisfy \Cref{conj-fold}. \qed
\end{corollary}

\section{Further remarks and perspectives}\label{further}
We recall the notation of \Cref{C1infty}. To an almost positive root of \(A_\infty\), we associate a monomial \(m_\alpha \in \Asub{\infty}\) as follows. We set
\begin{align*}
	m_{-\alpha_i} = \begin{cases}
		Y_{i,q^{\xi(i)+2}} & i \in (I_\infty)_0 \\
		Y_{i,q^{\xi(i)}} & i \in (I_\infty)_1
	\end{cases} && \text{and} && m_{\alpha_i} = \begin{cases}
		Y_{i,q^{\xi(i)}} & i \in (I_\infty)_0 \\
		Y_{i,q^{\xi(i)+2}} & i \in (I_\infty)_1
	\end{cases}
\end{align*}
for the negative and positive simple roots, and \(m_\alpha = \prod_{k=i}^{j}m_{\alpha_k}\) for a positive root \(\alpha = \sum_{k=i}^{j}\alpha_k\).
To make notation easier, for \(i \leq j\), we will denote by \(\alpha_{i,j}\) the root \(\sum_{k=i}^{j}\alpha_k\). When \(n\) is fixed, we will write \(\overline{\zeta} \coloneq \phi_{2n}(\zeta)\) for \(\zeta \in \mathcal{Y}_\infty\). Also, we will denote by \(\chi_q^\infty(m)\) the \(q\)-character of the simple \(\uqslinf\)-module \(L(m)\) and by \(\chi_q^\tor(\overline{m})\) the \(q\)-character of the simple \(\uqtor{2n}\)-module \(L(\overline{m})\). \par
For a given \(n\), using the geometric model of \Cref{folding-theorem} and \Cref{stable-arc}, it is not hard to see that \(\chi_q^\infty (m_{\alpha})\) is a cluster variable of \(\mathcal{A}_\infty\) belonging to an orbit-cluster if and only if one of the following happens:
\begin{enumerate}
	\item the root \(\alpha\) has odd length,
	\item the root \(\alpha\) has even length smaller than \(2n\). 
\end{enumerate}
It follows from \Cref{theorem-ice-folding} and \Cref{cor-ice-folding} that the cluster variables of \(\mathcal{A}_{2n,\tor}\) are exactly \(\chi_q^\tor (\overline{m_{\alpha}})\) for \(\alpha\) of this form. But what is \(\chi_q^\tor(\overline{m_{\alpha_{1,2n}}})\) is not clear from this description. We expect it behaves differently, because \(\overline{m_{\alpha_{1,2n}}}\) corresponds, under our description, to the imaginary root of \(A_{2n-1}^{(1)}\).
\begin{proposition}
	Suppose \(n \geq 3\). Then
	\begin{align*}
		\overline{\chi_q^\infty(m_{\alpha_{1,2n}})} = \chi_q^\tor(\overline{m_{\alpha_{1,2n}}}) + \chi_q^\tor(\overline{m_{\alpha_{3,2n-2}}})\cdot f,
	\end{align*}
	where \(f\) is some cluster monomial in frozen variables of \(\mathcal{A}_{2n,\tor}\).
\end{proposition}

\begin{proof}
	We write down a few standard exchange relations in \(\mathcal{A}_\infty\), with \(i + 2 \leq j\):
	\begin{align}
		\chi_q^\infty(m_{\alpha_{i,j}})\cdot\chi_q^\infty(m_{\alpha_{j+1}}) &= \chi_q^\infty(m_{\alpha_{i,j+1}}) + \chi_q^\infty(m_{\alpha_{i,j-2}})\cdot \chi_q^\infty(m_{-\alpha_{j+2}}) \cdot f, \label{exc1}\\
		\chi_q^\infty(m_{\alpha_{i-1}})\cdot\chi_q^\infty(m_{\alpha_{i,j}}) &= \chi_q^\infty(m_{\alpha_{i-1,j}}) + \chi_q^\infty(m_{\alpha_{i+2,j}})\cdot \chi_q^\infty(m_{-\alpha_{i-2}}) \cdot f, \label{exc2}\\
		\chi_q^\infty(m_{\alpha_{i,j}})\cdot\chi_q^\infty(m_{-\alpha_{j}}) &= \chi_q^\infty(m_{\alpha_{i,j-1}})\cdot f + \chi_q^\infty(m_{\alpha_{i,j-2}})\cdot \chi_q^\infty(m_{-\alpha_{j+1}}) \cdot f, \label{exc3}\\
		\chi_q^\infty(m_{-\alpha_{i}})\cdot\chi_q^\infty(m_{\alpha_{i,j}}) &= \chi_q^\infty(m_{\alpha_{i+1,j}})\cdot f + \chi_q^\infty(m_{\alpha_{i+2,j}})\cdot \chi_q^\infty(m_{-\alpha_{i-1}}) \cdot f. \label{exc4}
	\end{align}
	Here \(f\) stands for some cluster monomial in frozen variables. Applying \(\phi_{2n}\) to equations \eqref{exc1} with \(i=1\) and \(j=2n-1\) and \eqref{exc4} with \(i=1\) and \(j=2n-3\), we get
	\begin{align*}
		\chi_q^\tor(\overline{m_{\alpha_{1,2n-1}}})\cdot \chi_q^\tor(\overline{m_{\alpha_{2n}}}) = \overline{\chi_q^\infty(m_{\alpha_{1,2n}})} + \chi_q^\tor(\overline{m_{\alpha_{1,2n-3}}})\cdot\chi_q^\tor(L(\overline{m_{-\alpha_{2n+1}}}))\cdot \overline{f}
	\end{align*}
	and
	\begin{align*}
		\chi_q^\tor(\overline{m_{\alpha_{1,2n-3}}})\cdot \chi_q^\tor(\overline{m_{-\alpha_{1}}}) = \chi_q^\tor(\overline{m_{\alpha_{2,2n-3}}})\cdot\overline{f} + \chi_q^\tor(\overline{m_{\alpha_{3,2n-3}}})\cdot\chi_q^\tor(L(\overline{m_{-\alpha_{0}}}))\cdot \overline{f}.
	\end{align*}
	But as \(\chi_q^\tor(L(\overline{m_{-\alpha_{2n+1}}})) = \chi_q^\tor(L(\overline{m_{-\alpha_{1}}}))\), we can combine the two equations to have
	\begin{multline}\label{imchar}
		\chi_q^\tor(\overline{m_{\alpha_{1,2n-1}}})\cdot \chi_q^\tor(\overline{m_{\alpha_{2n}}}) = \\ \overline{\chi_q^\infty(m_{\alpha_{1,2n}})} + \chi_q^\tor(\overline{m_{\alpha_{2,2n-3}}})\cdot\overline{f} + \chi_q^\tor(\overline{m_{\alpha_{3,2n-3}}})\cdot\chi_q^\tor(L(\overline{m_{-\alpha_{0}}}))\cdot \overline{f}.
	\end{multline}
	A similar argument using \eqref{exc2} and \eqref{exc3} gives
	\begin{multline*}
		\chi_q^\tor(\overline{m_{\alpha_{1}}})\cdot\chi_q^\tor(\overline{m_{\alpha_{2,2n}}})  \\ = \overline{\chi_q^\infty(m_{\alpha_{1,2n}})} + \chi_q^\tor(\overline{m_{\alpha_{4,2n-1}}})\cdot\overline{f} + \chi_q^\tor(\overline{m_{\alpha_{4,2n-2}}})\cdot\chi_q^\tor(L(\overline{m_{-\alpha_{1}}}))\cdot \overline{f}.
	\end{multline*}
	\emergencystretch 3em The \(q\)-characters \(\chi_q^\tor(\overline{m_{\alpha_{2,2n-3}}})\cdot\overline{f}\), \(\chi_q^\tor(\overline{m_{\alpha_{3,2n-3}}})\cdot\chi_q^\tor(L(\overline{m_{-\alpha_{0}}}))\cdot \overline{f}\), \(\chi_q^\tor(\overline{m_{\alpha_{4,2n-1}}})\cdot\overline{f}\) and \(\chi_q^\tor(\overline{m_{\alpha_{4,2n-2}}})\cdot\chi_q^\tor(L(\overline{m_{-\alpha_{1}}}))\cdot \overline{f}\) are all \(q\)-characters of pairwise non-isomorphic simple \(\uqtor{2n}\)-modules (\Cref{tor-cluster}). We deduce that \(\overline{\chi_q^\infty(m_{\alpha_{1,2n}})}\) is the \(q\)-character of an actual \(\uqtor{2n}\)-module. Certainly, \(\overline{\chi_q^\infty(m_{\alpha_{1,2n}})}\) contains \(\chi_q^\tor(\overline{m_{\alpha_{1,2n}}})\). Let us find all its composition factors. \par
	By computing \eqref{exc2} with \(i = 1\) and \(j = 2n - 1\), and \eqref{exc3} with \(i=3\) and \(j=2n-1\), we can obtain (similar to the computation before)
	\begin{multline*}
		\chi_q^\tor(\overline{m_{\alpha_{1,2n-1}}})\cdot \chi_q^\tor(\overline{m_{\alpha_{2n}}}) = \\ \overline{\chi_q^\infty(m_{\alpha_{0,2n-1}})} + \chi_q^\tor(\overline{m_{\alpha_{3,2n-2}}})\cdot\overline{f} + \chi_q^\tor(\overline{m_{\alpha_{3,2n-3}}})\cdot\chi_q^\tor(L(\overline{m_{-\alpha_{0}}}))\cdot \overline{f}.
	\end{multline*}
	Comparing this expression with \eqref{imchar}, we can deduce that \(\chi_q^\tor(\overline{m_{\alpha_{3,2n-2}}})\cdot\overline{f}\) is a summand of \(\overline{\chi_q^\infty(m_{\alpha_{1,2n}})}\). Similar to the previous arguments, we can obtain the equations
	\begin{multline*}
		\overline{\chi_q^\infty(m_{\alpha_{1,2n}})}\cdot \chi_q^\tor(\overline{m_{-\alpha_1}}) = \\ \chi_q^\tor(\overline{m_{\alpha_{2,2n}}})\cdot \overline{f} +  \chi_q^\tor(\overline{m_{\alpha_{3,2n-1}}})\cdot \overline{f} +  \chi_q^\tor(\overline{m_{\alpha_{3,2n-2}}})\cdot \chi_q^\tor(\overline{m_{-\alpha_1}}) \cdot \overline{f}
	\end{multline*}
	and
	\begin{multline*}
		\overline{\chi_q^\infty(m_{\alpha_{1,2n}})}\cdot \chi_q^\tor(\overline{m_{-\alpha_{2n}}}) = \\ \chi_q^\tor(\overline{m_{\alpha_{1,2n-1}}})\cdot \overline{f} +  \chi_q^\tor(\overline{m_{\alpha_{2,2n-2}}})\cdot \overline{f} +  \chi_q^\tor(\overline{m_{\alpha_{3,2n-2}}})\cdot \chi_q^\tor(\overline{m_{-\alpha_{2n}}}) \cdot \overline{f}.
	\end{multline*}
	We can deduce from that that \(\overline{\chi_q^\infty(m_{1,2n})}\) has at most three summands. If it had exactly three, this would implies
	\begin{align*}
		\chi_q^\tor(\overline{m_{-\alpha_{2n}}})\cdot \chi_q^\tor(\overline{m_{\alpha_{3,2n-1}}})\cdot \overline{f} =  \chi_q^\tor(\overline{m_{-\alpha_1}})\cdot  \chi_q^\tor(\overline{m_{\alpha_{2,2n-2}}})\cdot \overline{f},
	\end{align*}
	which is absurd. This completes the proof.
\end{proof}
Similar computations give
\begin{fact}
	Let \(i,j \in \mathbb{Z}\) be such that \(j-i\leq 4n-2\). Then \(\overline{\chi_q^\infty(m_{\alpha_{i,j}})}\) is the \(q\)-character of an actual \(\uqtor{2n}\)-module. In particular, the simple \(\uqslinf\)-module \(L(m_{\alpha_{i,j}})\) satisfy \Cref{conj-fold}.
\end{fact}

However, it is not clear (to the author) if it is possible to deduce the same result for \(\alpha\) of length \(4n\), that is, for \(j-i = 4n-1\). \par
In Example 6.15 of \cite{N11}, Nakajima explains that for \(n=1\), the module \(L(\overline{m_{1,2n}})\) is not real, which in \(q\)-character theory amounts to
\begin{align*}
	\chi_q^\tor(\overline{m_{1,2n}})^2 \neq \chi_q^\tor(\overline{m_{1,4n}}).
\end{align*}
This raises the question whether this is a general phenomenon.
\begin{conjecture}\label{imag-conj}
	The simple \(\uqtor{2n}\)-module \(L(\overline{m_{1,2n}})\) is an imaginary module, that is, \(L(\overline{m_{1,2n}})\ast L(\overline{m_{1,2n}})\) is not simple. 
\end{conjecture}
Although some examples of imaginary modules of affinized quantum Kac-Moody algebras are known, for example \cite{Lec03,BriCha,BriMou}, these are hard to find. Proving this conjecture would give a new interesting example for quantum toroidal algebras. \par
The difficulty of computing \(\chi_q^\tor(\overline{m_{1,4n}})\) does not allow us to prove or disprove this conjecture directly, in terms of the arguments above. For \(n \geq 3\) we do manage to show 
\begin{align*}
	\chi_q^\tor(\overline{m_{1,2n}})^2 = \chi_q^\tor(\overline{m_{1,4n}}) + \lambda f 
\end{align*}
with \(\lambda \in \{0,1,2\}\). However, \(\lambda = 0\) if and only if \(\overline{\chi_q^\infty(m_{1,4n})}\) is a virtual \(q\)-character. We can deduce
\begin{claim}
	Assume \Cref{conj-fold} holds. Then \(L(\overline{m_{1,2n}})\) is an imaginary \(\uqtor{2n}\)-module for \(n\geq 3\).
\end{claim}

\section*{References}
\addcontentsline{toc}{section}{References}
\emergencystretch 3em
\printbibliography[heading=none]

\bigskip
\bigskip

Lior Silberberg, \textsc{Université Paris Cité, Sorbonne Université, CNRS, IMJ-PRG, F-75013 Paris, France} \par
\textit{E-mail address}: \texttt{lior.silberberg@imj-prg.fr}

\end{document}